\documentclass[12pt,a4paper]{article}
\usepackage[utf8]{inputenc}
\usepackage[english]{babel}
\usepackage[ruled,section]{algorithm}
\usepackage{makeidx}
\usepackage{amsmath}
\usepackage{amsthm}
\usepackage{amsfonts}
\usepackage{amssymb}
\usepackage{mathrsfs}
\usepackage{graphicx}
\usepackage{fancyhdr}
\usepackage[margin=25mm]{geometry}
\usepackage{simplewick}
\usepackage{latexsym}
\usepackage[all]{xy}
\usepackage{enumerate}
\usepackage[plainpages=false]{hyperref}
\usepackage{hyperref}
\usepackage{color}
\usepackage{esint}
\usepackage{authblk}
\usepackage{perpage}
\usepackage[font=small]{caption}
\makeindex
\hypersetup{pdfpagelabels,hyperindex,colorlinks=true,breaklinks=true,bookmarks=true,linkcolor=black,citecolor=black,
urlcolor=black}
\newcommand{\real}{\mathbb{R}}
\newcommand{\n}{\mathbb{N}}

\newcommand{\rn}{ {\mathbb{R}^n} }
\renewcommand{\d}{\; \mathrm{d}}
\numberwithin{equation}{section}
\newtheorem{teo}{Theorem}[section]
\newtheorem{lem}[teo]{Lemma}
\newtheorem{rmk}[teo]{Remark}

\newtheorem{defin}[teo]{Definition}

\newtheorem{prop}[teo]{Proposition}

\newtheorem{claim}[teo]{Claim}

\begin{document}

\title{\bf\Large $C^{1,\alpha}$
regularity for fully nonlinear elliptic equations with superlinear growth in the gradient}
\author[1]{Gabrielle Nornberg\footnote{Corresponding author. Email address: gabrielle@icmc.usp.br}\footnote{The work was supported by Capes PROEX/PDSE grant 88881.134627/2016-01.}}
\affil[1]{Pontifícia Universidade Católica do Rio de Janeiro, Brazil}
\date{}

\maketitle

{\small\noindent{\bf{Abstract.}} We extend the Caffarelli-Świech-Winter $C^{1,\alpha}$ regularity estimates to $L^p$-viscosity solutions of fully nonlinear uniformly elliptic equations in nondivergence form with superlinear growth in the gradient and unbounded coefficients. As an application, in addition to the usual $W^{2,p}$ results, we prove the existence of positive eigenvalues for proper operators with nonnegative unbounded weight, in particular for Pucci's operators with unbounded coefficients.

\medskip

{\small\noindent{\bf{Résumé.}} Dans cet article on étend les résultats de régularité $C^{1, \alpha}$ de Caffarelli-Świech-Winter aux solutions de $ L^p$ viscosité des équations complètement non-linéaires, uniformément elliptiques, sous forme non-divergence, avec croissance super-linéaire par rapport au gradient et coefficients non bornés. Dans le cadre d'une application, en plus des résultats habituels $W ^{2, p}$, on prouve l’existence de valeurs propres positives pour les opérateurs propres avec poids non borné non négatif, en particulier pour les opérateurs de Pucci à coefficients non bornés.

\medskip

{\small\noindent{\bf{Keywords.}} {Regularity, Estimates, Superlinear gradient growth, Nondivergence form.}

\medskip

{\small\noindent{\bf{MSC2010.}} {35J15, 35B65, 35J60, 	35P30, 35P15.}

\section{Introduction}\label{Introduction}

The seminal work of Caffarelli \cite{Caf89} in 1989 brought an innovative approach of looking at Schauder type results via iterations from the differential quotients that are perturbations of solutions of the respective autonomous equations.
The techniques in \cite{Caf89}, which contains in particular $C^{1,\alpha}$ estimates for $L^p$-viscosity solutions of uniformly elliptic equations $F(x,D^2u)=f(x)$, allowed Świech \cite{Swiech} to extend them to more general operators $F(x,u,Du,D^2u)$ and later Winter \cite{Winter} to boundary and global bounds. However, everything that is available in the literature, to our knowledge, for $L^p$-viscosity solutions in the fully nonlinear framework, concerns only structures with either linear gradient growth or bounded coefficients, except for some particular cases of extremal equations with small coefficients, see \cite{KSexist2009}.
It is our goal here to obtain $C^{1,\alpha}$ regularity and estimates for general fully nonlinear uniformly elliptic equations, with at most quadratic growth in the gradient and unbounded coefficients.

The study of such quasilinear elliptic equations with quadratic dependence in the gradient had its beginning in the '80s, essentially with the works of Boccardo, Murat and Puel \cite{BMP2}, \cite{BMP1} and became a relevant research topic which still develops. This type of nonlinearity often appears in risk-sensitive stochastic problems, as well as in large deviations, control and game theory, mean-fields problems. Moreover, the class of equations in the form $Lu = g (x,u,D u)$, where $L$ is a second order general operator and $g$ has quadratic growth in the gradient, is invariant under smooth changes of the function $u$ and the variable $x$. Due to this fact, this class is usually referred as having \textit{natural} growth in the gradient.

Rather complete $C^\alpha$ regularity results for fully nonlinear uniformly elliptic equations with up to quadratic growth in the gradient were obtained in \cite{arma2010}, in the most general setting of unbounded coefficients, for $L^p$-viscosity solutions. Then, the question of $C^{1,\alpha}$ regularity for the same class arises naturally. In the present work we show, as can be expected, that $C^{1,\alpha}$ regularity and estimates are valid in this context. These $C^{1,\alpha}$ estimates are instrumental in the recent study of multiplicity for nonproper equations in \cite{multiplicidade}.

We note that Trudinger, independently from \cite{Caf89}, in  \cite{T89} proved $C^{1,\alpha}$ regularity in a less general scenario than Świech and Winter, under a continuity hypothesis for $F$, dealing with $C$-viscosity solutions and approximations under supconvolutions. In that paper, it was stated that a priori estimates for solutions in $C^{1,\alpha}$ of superlinear equations could be derived from the arguments in \cite{T89} and \cite{T88}. However, the question of regularity is more complicated (for a discussion on differences between a priori bounds and regularity results we refer to \cite{SB}).

We also quote some other papers on $C^{1,\alpha}$ regularity, the classical works \cite{KrylovBook}, \cite{LU}, \cite{LU89} for linear equations, \cite{MilSilv} for Neumann boundary conditions, \cite{ST} for asymptotically convex operators, \cite{IS} (local) and \cite{BDC1beta} (global) for degenerate elliptic operators,
\cite{CKS} and \cite{Krylov} for parabolic equations possibly with VMO coefficients.
Furthermore, Wang \cite{W2} has made an important contribution to $C^{1,\alpha}$ regularity for the parabolic equation $u_t+F(x,D^2 u)=g(t,x,Du)$, where $|g(t,x,p)|\leq A\, |p|^2+g(t,x)$, for bounded coefficients, see lemma 1.6 in \cite{W2} (which uses theorem 4.19 in \cite{W1}). Sharp regularity results for general parabolic equations with linear gradient growth can be found in \cite{JE}, and very complete $C^{1,\alpha}$ estimates on the boundary for solutions in the so called $S^*$-class for equations with linear gradient growth and unbounded coefficients in \cite{DEWboundary}.

It is also essential to mention an important series of papers due to Koike and Świech  \cite{KSweakharnack}, \cite{KSmp2004}, \cite{KSmpite2007}, \cite{KSexist2009}, in which they proved ABP and weak Harnack inequalities for $L^p$-viscosity solutions of equations with superlinear growth in the gradient, together with several theorems about existence, uniqueness and $W^{2,p}$ estimates  for solutions of extremal equations involving Pucci's operators with unbounded coefficients, see in particular theorem 3.1 in \cite{KSexist2009}. Many of our arguments depend on the machinery in these works.
\medskip

Next we list our hypotheses.
For $F(x,r,p,X)$ measurable in $x$, we consider the general structure condition
\begin{align} \label{SCmu}
\mathcal{M}_{\lambda, \Lambda}^- (X-Y)-b(x)|p-q|-\mu |p-q|(|p|+|q|)-d(x)\,\omega (|r-s|) \nonumber \\
\leq F(x,r,p,X) - F(x,s,q,Y) \tag*{$(SC)^\mu$} \\
\leq \mathcal{M}_{\lambda, \Lambda}^+ (X-Y)+b(x)|p-q|+\mu |p-q|(|p|+|q|)+d(x)\,\omega (|r-s|) \; \textrm{ for } x\in \Omega\nonumber
\end{align}
where $F(\cdot,0,0,0)\equiv 0$ and $0<\lambda \leq \Lambda$, $b\in L^p_+ (\Omega)$ for some $ p>n$, $d\in L^\infty_+ (\Omega)$, $\mu\geq 0$ and $\omega$ is a modulus of continuity (see  section \ref{Preliminaries}).

In order to measure the oscillation of $F$ in the $x$ entry, we define, as in \cite{Caf89}, \cite{Winter},
\begin{align} \label{def beta}
\beta(x,x_0)=\beta_F (x,x_0):= \sup_{X\in \mathbf{S}^n\setminus \{0\} } \frac{|F(x,0,0,X)-F(x_0,0,0,X)|}{\|X\|} .
\end{align}
Notice that $\beta$ is a bounded function by \ref{SCmu} and consider the usual hypothesis, as in \cite{Caf89}, \cite{Winter}: given $\theta>0$, there exists $r_0=r_0\, (\theta)>0$ such that
\begin{align}\label{Htheta} \tag{$H_{\theta}$}
\;\left( \frac{1}{r^n} \int_{B_r(x_0)\cap\Omega} \beta (x,x_0)^p \d x \right)^{\frac{1}{p}} \leq \theta \, ,\;\textrm{ for all } r\leq r_0 .
\end{align}

The following is our main result. To simplify its statement, here we assume $\omega (r)\leq \omega(1)r$ for all $r\geq 0$.

\begin{teo}
\label{C1,alpha regularity estimates geral}
Assume $F$ satisfies \ref{SCmu}, $f  \in L^p (\Omega)$, where $p>n$, and $\Omega\subset\rn$ is a bounded domain.
Let $u$ be an $L^p$-viscosity solution of
\begin{align}\label{F=f}
F(x,u,Du,D^2 u)=f(x) \quad \textrm{ in }\;\Omega
\end{align}
with $\| u \|_{L^{\infty} (\Omega)} + \|f \|_{L^p (\Omega)} \leq C_0$.
Then, there exists $\alpha\in (0,1)$ and $\theta=\theta (\alpha)$, depending on $n,p,\lambda,\Lambda,\|b\|_{L^p(\Omega)}$, such that if \eqref{Htheta} holds for all $r\leq \min \{ r_0, \mathrm{dist} (x_0,\partial\Omega)\}$, for some $r_0>0$ and for all $x_0 \in \Omega$, this implies that $u\in C^{1,\alpha}_{\mathrm{loc}} (\Omega)$ and for  any subdomain $\Omega^\prime \subset\subset \Omega$,
\begin{align}\label{estim local}
\|u\|_{C^{1,\alpha}(\overline{\Omega^\prime})} \leq C \,\{ \| u \|_{L^{\infty} (\Omega)}  + \|f \|_{L^p (\Omega)} \}
\end{align}
where $C$ depends only on $\,r_0,n,p,\lambda,\Lambda,\alpha,  \mu, \| b \|_{L^p (\Omega)},\omega(1)\| d \|_{L^\infty (\Omega)},\mathrm{diam} (\Omega)$, $\mathrm{dist} (\Omega^\prime,\partial\Omega), C_0$.
\vspace{0.02cm}

If in addition, $\partial\Omega\in C^{1,1}$ and $u\in C(\overline{\Omega})\cap C^{1,\tau} (\partial \Omega )$ is such that and $\| u \|_{L^{\infty} (\Omega)} + \|f \|_{L^p (\Omega)} + \|u\|_{C^{1,\tau} (\partial\Omega)} \leq C_1$, then there exists $\alpha\in (0,\tau)$ and $\theta =\theta (\alpha)$, depending on $n,p,\lambda , \Lambda , \|b\|_{L^p(\Omega)}$, so that if \eqref{Htheta} holds for some $r_0>0$ and for all $x_0 \in \overline{\Omega}$, this implies that $u\in C^{1,\alpha}(\overline{\Omega})$ and satisfies the estimate
\begin{align}\label{estim global}
\|u\|_{C^{1,\alpha}(\overline{\Omega})} \leq C\, \{ \| u \|_{L^{\infty} (\Omega)} + \|f \|_{L^p (\Omega)}  + \|u\|_{C^{1,\tau} (\partial\Omega)} \}
\end{align}
where $C$ depends on $\,r_0,n,p,\lambda,\Lambda,\alpha,\mu$,$\| b \|_{L^p (\Omega)},\omega(1)\| d \|_{L^\infty (\Omega)},\mathrm{diam} (\Omega),C_1$ and on the $C^{1,1}$ diffeomorphisms that describe the boundary.
\vspace{0.01cm}

If $\mu=0$, then the constant $C$ does not depend on $C_0$, $C_1$.
\end{teo}

We also consider, as in \cite{Swiech} and chapter 8 in \cite{CafCab}, a slightly different (smaller) version of $\beta$,
\begin{align} \label{def beta bar}
\bar{\beta}(x,x_0)=\bar{\beta}_F (x,x_0):= \sup_{X\in \mathbf{S}^n} \frac{|F(x,0,0,X)-F(x_0,0,0,X)|}{\|X\|+1} .
\end{align}
Consider  the hypothesis ${(\overline{H})}_\theta$, which is \eqref{Htheta} with $\beta$ replaced by $\bar{\beta}$. This hypothesis is trivially satisfied if $F(x,0,0,X)$ is uniformly continuous in $x$.

\begin{rmk}
If $\omega$ is an arbitrary modulus, we still have regularity and estimates for bounded~$b$, with the same dependence on constants as before, by adding 1 on the right hand side of \eqref{estim local} and \eqref{estim global}. 
In this case, we can replace \eqref{Htheta} by ${(\overline{H})}_\theta$ in Theorem \ref{C1,alpha regularity estimates geral}, see remark \ref{Remark qualquer modulo}.

Of course, explicit zero order unbounded terms that only depend on $u$ and $x$, can always be handled as being part of the right hand side $f(x)$.
\end{rmk}

\begin{rmk}
If $\mu = 0$ we can also obtain Theorem \ref{C1,alpha regularity estimates geral} in terms of ${(\overline{H})}_\theta$, see remark \ref{remark mu=0}.
\end{rmk}

The proof of Theorem \ref{C1,alpha regularity estimates geral} is based on Caffarelli's iteration method. Compared to \cite{Swiech}, \cite{Winter}, we use a simplified rescaling of variable which allows us to carry out the proof, without needing to use a twice differentiability property of viscosity solutions (whose validity is unknown for unbounded coefficients).
We also use ideas of Wang to deal with superlinear terms.

The structure of the paper is as follows. In section \ref{Preliminaries} we recall some known results which are used along the text. In section \ref{proof main th} we give a detailed proof of theorem \ref{C1,alpha regularity estimates geral}, splitting it into local and boundary parts. The final sections \ref{W2,p regularity} and \ref{First eigenvalue} are devoted to applications. Section \ref{W2,p regularity} deals with $W^{2,p}$ regularity -- see theorem \ref{W2,p quad} for the main regularity result; we also present a generalized Nagumo's lemma \ref{Nagumo}.
Section \ref{First eigenvalue} is related to existence of eigenvalues for general operators with a nonnegative unbounded weight, see theorem \ref{exist eig for F c geq 0} (these results play an important role in \cite{multiplicidade}).

Significant contributions on eigenvalues of continuous operators in nondivergence form in bounded domains include the fundamental work \cite{BNV} for linear operators; \cite{QB} for convex fully nonlinear operators; \cite{MJ} for nonlocal operators;  \cite{BD}, \cite{BDru}, \cite{L83}, and the recent \cite{BDPR} for degenerate elliptic operators. Theorem~\ref{exist eig for F c geq 0} is a slight improvement to the general existence theory about nonconvex operators possessing first eigenvalues in \cite{Arms2009} (see also \cite{IY}), since we are not supposing that our nonlinearity is uniformly continuous in $x$.

If, in addition, we have $W^{2,p}$ regularity of solutions, we can extend theorem \ref{exist eig for F c geq 0} even further, allowing an unbounded first order coefficient. Eigenvalues for fully nonlinear operators with such coefficients have been previously studied, to our knowledge,  only for radial operators and eigenfunctions, in \cite{II1} and \cite{II2}.
As a particular case of theorem \ref{exist eig for F c geq 0}, we obtain the existence of positive eigenvalues with a nonnegative unbounded weight for the extremal Pucci's operators with unbounded coefficients.

\begin{prop}\label{corolPucci}
Let $\Omega\subset\rn$ a bounded $C^{1,1}$ domain, $b,\, c\in L^p_+(\Omega)$, $c\gneqq 0$, for $p>n$. Then, there exists $\varphi_1^\pm\in W^{2,p}(\Omega)$ such that, for $\lambda_1^{\pm}$ defined in section \ref{First eigenvalue}, we have $\lambda_1^{\pm}>0$ and
\begin{align}
\left\{
\begin{array}{rclcc}
\mathcal{M}^\pm_{\lambda,\Lambda} (D^2 \varphi_1^\pm )\pm b(x)|D \varphi_1^\pm |+\lambda_1^\pm c(x) \varphi_1^\pm &=& 0 &\mbox{in} & \Omega \\
\varphi_1^\pm &>& 0 &\mbox{in} &\Omega \\
\varphi_1^\pm &=& 0 &\mbox{on} &\partial\Omega.
\end{array}
\right.
\end{align}
\end{prop}

\textit{Acknowledgments.} I would like to thank my Ph.D. advisor, Professor Boyan Sirakov, for years of patient guidance, many valuable discussions and helpful suggestions.

I am also deeply indebted to Prof.\ Diego Moreira, whose valuable remarks and suggestions led to a substantial improvement of the paper, in particular by clarifying some arguments.

\section{Preliminaries}\label{Preliminaries}

We start detailing the hypothesis \ref{SCmu}. Notice that the condition over the highest order term $X$, for $p=q$ and $ r=s$, implies that $F$ is a uniformly elliptic operator.
In \ref{SCmu},
$$
\mathcal{M}^+_{\lambda,\Lambda}(X):=\sup_{\lambda I\leq A\leq \Lambda I} \mathrm{tr} (AX)\;\;\textrm{ and} \quad \mathcal{M}^-_{\lambda,\Lambda}(X):=\inf_{\lambda I\leq A\leq \Lambda I} \mathrm{tr} (AX)
$$
are the Pucci's extremal operators.  See, for example, \cite{CafCab} and \cite{QB} for their properties.

By \textit{modulus} we mean a function $\omega :[0,+\infty] \rightarrow [0,+\infty]$ continuous at $0$ with $\omega (0)=0$. We may consider $\omega$ increasing and continuous, up to replacing it by a larger function. We can also suppose $\omega$ subadditive, from where $\omega (k)\leq (k+1)\,\omega (1)$ for all $k\geq 0$.  
Moreover, we may assume that $d$ is bounded everywhere, up to defining it in a zero measure set. 
We stress that our results are not restricted to bounded zero order terms, since unbounded ones which only depend on $x$ and $u$ can always be treated as being part of the right hand side.

Next we recall the definition of $L^p$-viscosity solution.

\begin{defin}\label{def Lp-viscosity sol}
Let $f\in L^p_{\textrm{loc}}(\Omega)$.
We say that $u\in C(\Omega)$ is an $L^p$-viscosity subsolution $($respectively, supersolution$)$ of \eqref{F=f} if whenever $\phi\in  W^{2,p}_{\mathrm{loc}}(\Omega)$, $\varepsilon>0$ and $\mathcal{O}\subset\Omega$ open are such that
\begin{align*}
F(x,u(x),D\phi(x),D^2\phi (x))-f(x)  \leq -\varepsilon\;\;
( F(x,u(x),D\phi(x),D^2\phi (x))-f(x)  \geq \varepsilon )
\end{align*}
for a.e. $x\in\mathcal{O}$, then $u-\phi$ cannot have a local maximum $($minimum$)$ in $\mathcal{O}$.
\end{defin}

We can think about $L^p$-viscosity solutions for any $p>\frac{n}{2}$, since this restriction makes all test functions $\phi\in W^{2,p}_{\mathrm{loc}}(\Omega)$ continuous and having a second order Taylor expansion \cite{CCKS}.
We are going to deal mostly with the case $p>n$. In particular, for $\Omega$ bounded with $\partial\Omega\in C^{1,1}$, this implies that the continuous injection $W^{2,p}(\Omega)\subset C^1(\overline{\Omega})$ is compact, for all $n\geq 1$.
Further, when $p>n$ and $F$ possesses the quadratic structure \ref{SCmu}, the maximum or minimum in the definition \ref{def Lp-viscosity sol} can taken to be strict (see for example proposition 1 in \cite{KoikePerronRevisited}).

If $F$ and $f$ are continuous in $x$, we can use the more usual notion of \textit{$C$-viscosity} sub and supersolutions, as in \cite{user}.
Both definitions are equivalent when, moreover, $F$ satisfies \ref{SCmu} for bounded $b$, with $\mu,d\equiv 0$ and $p\geq n$, by proposition 2.9 in \cite{CCKS}; we will be using them interchangeably, in this case, throughout the text.

A \textit{strong} sub or subsolution belongs to $W^{2,p}_{\mathrm{loc}}(\Omega)$ and satisfies the inequality at almost every point. Such notions are related, up to quadratic growth, as shows the next proposition.
\begin{prop}\label{Lpiffstrong.quad}
Assume $F$ satisfies \ref{SCmu} with $b\in L^q_+(\Omega)$, $q\geq p\geq n$, $q>n$ and $f\in L^p(\Omega)$.
Then, $u\in W^{2,p}_{\mathrm{loc}}(\Omega)$ is a strong subsolution $($supersolution$)$ of $F=f$ in $\Omega$ if and only if it is an $L^p$-viscosity subsolution $($supersolution$)$ of it.
\end{prop}

See theorem 3.1 and proposition 9.1 in \cite{KSweakharnack} for a proof, even for more general conditions on $\mu$ and the exponents $p,q$.
A \textit{solution} is always both sub and supersolution of the equation.

The next proposition follows from theorem 4 in \cite{arma2010} in the case $p=n$.
For a version with more general exponents and coefficients, we refer to proposition 9.4 in \cite{KSweakharnack}.

\begin{prop}  \label{Lpquad} {$($Stability$)$}
Let $F$, $F_k$ operators satisfying \ref{SCmu},\,$b\in L^q_+(\Omega), q\geq p\geq n, q>n$, $f, \, f_k\in L^p(\Omega)$. Let $u_k\in C(\Omega)$ be an $L^p$-viscosity subsolution $($supersolution$)$ of $F_k=f_k$ i.e.
$$
F_k(x,u_k,Du_k,D^2u_k)\geq f_k(x) \;\;\;\textrm{in} \;\;\;\Omega\quad (\leq)\quad\textrm{ for all }\; k\in \n .
$$
Suppose  $u_k\rightarrow u$ in $L_{\mathrm{loc}}^\infty  (\Omega)$ as $k\rightarrow \infty$ and for each ball $B\subset\subset \Omega$ and $\varphi\in W^{2,p}(B)$, setting
\begin{align*}
g_k(x):=F_k(x,u_k,D\varphi,D^2\varphi)-f_k(x) ; \;\;
g(x):=F(x,u,D\varphi,D^2\varphi)-f(x),
\end{align*}
we have $\| (g_k-g)^+\|_{L^p(B)}$ $(\| (g_k-g)^-\|_{L^p(B)}) \rightarrow 0$ as $k\rightarrow \infty$. Then $u$ is an $L^p$-viscosity subsolution $($supersolution$)$ of $F=f$ i.e.
$F(x,u,Du,D^2u)\geq f(x)$  \;$ (\leq)\;$ in $\Omega$.
\vspace{0.01cm}

If $F$ and $f$ are continuous in $x$, then it is enough that the above holds for every $\varphi\in C^2 (B)$, in which case $u$ is a $C$-viscosity subsolution $($supersolution$)$ of $F=f$ in $\Omega$.
\end{prop}

\begin{rmk} \label{Lpquadencaixados}
Proposition \ref{Lpquad} is valid if we have $f_k \in L^p(\Omega_k)$, $u_k\in C(\Omega_k)$, for an increasing sequence of domains $\Omega_k \subset \Omega_{k+1}$ such that $\Omega :=\bigcup_{k\in \n} \Omega_k\,$, see proposition 1.5 in \cite{Winter}.
\end{rmk}

Denote $\mathcal{L}^\pm[u]:=\mathcal{M}^\pm_{\lambda,\Lambda} (D^2u)\pm b(x)|Du|$, where $b\in L^p_+(\Omega)$, and $F[u]:=F(x,u,Du,D^2u)$.

We recall Alexandrov-Bakelman-Pucci type results with unbounded ingredients and quadratic growth, which will be referred simply by ABP.

\begin{prop} \label{wABPquad}
Let $\Omega$ bounded, $\mu\geq 0$, $b\in L^q_+ (\Omega)$ and $f\in L^p (\Omega)$, for $q\geq p \geq n$, $q>n$.
Then, there exist $\delta=\delta (n,p,\lambda, \Lambda, \mathrm{diam}(\Omega), \|b\|_{L^q(\Omega)})>0$ such that if
\begin{align*}
\quad\quad\quad\quad \mu \|f^-\|_{L^p(\Omega)} \, (\mathrm{diam}(\Omega))^{\frac{n}{p}} \leq \delta  \quad\quad\quad (\mu \|f^+\|_{L^p(\Omega)}  \, (\mathrm{diam}(\Omega))^{\frac{n}{p}} \leq \delta )
\end{align*}
then every $u\in C(\overline{\Omega})$ which is an $L^p$-viscosity subsolution $($supersolution$)$ of
\begin{align*}
\mathcal{L}^+[u]+\mu |D u|^2 \geq f(x)\;\; \mathrm{in}\;\; \Omega\cap\{u>0\} \quad
\left( \mathcal{L}^-[u]-\mu |D u|^2 \leq f(x)\;\; \mathrm{in}\;\; \Omega\cap\{u<0\}  \,\right)
\end{align*}
satisfies, for a constant $C_A$ depending on $n,p,\lambda,\Lambda, \|b\|_{L^q(\Omega)},\mathrm{diam}(\Omega)$, the estimate
\begin{align*}
\max_{\overline{\Omega}} u \leq \max_{\partial \Omega} u +C_A \,  \|f^-\|_{L^p(\Omega)} \quad \left( \min_{\overline{\Omega}} u \geq \min_{\partial \Omega} u -C_A  \|f^+\|_{L^p(\Omega)} \right).
\end{align*}
Moreover, $C_A$ remains bounded if these quantities are bounded.
\end{prop}

As a matter of fact, ABP is valid under more general conditions, even for unbounded $\mu$. We refer to theorem 2.6 and lemma 9.3 in \cite{KSweakharnack}, and theorem 3.4 in \cite{Naka}, for a precise dependence on constants (see also \cite{KSmpite2007} and \cite{KSexist2009}).
For a simplified proof in the case where $\mu> 0$ is constant and $p>n$ (which is the only superlinear case that we need along the text) we also refer to \cite{tese}.

\begin{prop}{$(C^\beta$ Regularity\,$)$}
\label{Cbetaquad}
Assume $F$ satisfies \ref{SCmu} for $N=0$, $q=0$, $s=0$ and $b\in L^q_+(\Omega)$, for $q\geq p \geq n$, $q>n$. Let $u\in C(\Omega)$ be an $L^p$-viscosity solution of \eqref{F=f} with $f\in L^p(\Omega)$. Then there exists $\beta\in (0,1)$ depending on $n,p,\lambda, \Lambda$ and $ \|b\|_{L^q(\Omega)}$ such that $u\in C^{\beta}_{\textrm{loc}}(\Omega)$ and for any subdomain $\Omega^\prime \subset\subset\Omega$ we have
\begin{align*}
\|u\|_{C^\beta(\Omega^\prime)}\leq K_1 \,\{\|u\|_{L^\infty (\Omega)} + \|f\|_{L^p(\Omega)}+\|d\|_{L^p(\Omega)}\,\omega(\|u\|_{L^\infty (\Omega)}) \}
\end{align*}
where $K_1$ depends only on $n,p,\lambda,\Lambda,\mu, \|b\|_{L^q(\Omega)}, \|u\|_{L^\infty (\Omega^\prime)}, \mathrm{dist}(\Omega^\prime, \partial\Omega)$.

If, in addition, $u\in C (\overline{\Omega})\cap C^\tau (\partial\Omega)$ and $\Omega$ satisfies a uniform exterior cone condition with size $L$, then there exists $\beta_0= \beta_0 (n,p,\lambda, \Lambda, L,\|b\|_{L^q(\Omega)}) \in (0,1)$ and $\beta = \min (\beta_0, \frac{\tau}{2})$ such that
\begin{align*}
\|u\|_{C^\beta(\overline{\Omega})}\leq K_1 \,\{\|u\|_{L^\infty (\Omega)} + \|f\|_{L^p(\Omega)} + \|u\|_{C^\tau (\partial\Omega)} +\|d\|_{L^p(\Omega)}\,\omega(\|u\|_{L^\infty (\Omega)})\}
\end{align*}
where $K_1$ depends on $n,p,\lambda,\Lambda,\mu, L, \|b\|_{L^q(\Omega)},\omega (1)\|d\|_{L^p(\Omega)} ,\mathrm{diam} (\Omega),\|u\|_{L^\infty (\Omega)}$.
In both cases, $K_1$ remains bounded if these quantities are bounded.

The same result holds if, instead of a solution of \eqref{F=f}, $u$ is only an $L^p$-viscosity solution of the inequalities
$\mathcal{L}^-[u]-\mu |Du|^2\leq g(x)$ and $ \mathcal{L}^+[u]+\mu |Du|^2\geq -g(x)$ in $\Omega$.

If $\mu=0$, the final constant does not depend on a bound from above on $\|u\|_{L^\infty (\Omega)}$.
\end{prop}

\begin{proof} This is a direct consequence of the proof of theorem 2 in \cite{arma2010}, reading the $L^n$-viscosity sense there as $L^p$-viscosity one, changing $b\in L^p, \, d,f\in L^n$ there by $b\in L^q, \,d,f\in L^p$. The corresponding growth lemmas and exponents concerning $\rho$ must be replaced by $\rho^{1-\frac{n}{p}}$, which appear by using proposition \ref{wABPquad} (for $\mu=0$) instead of theorem 3 there.

The zero order term is handled as being part of the right hand side, since the whole proof is valid if we only have $u$ as an $L^p$-viscosity solution of inequalities $\mathcal{L}^+ [u]\geq - g(x)$ and $\mathcal{L}^-[u]\leq g(x)$ in the case $\mu=0$ (see the final remark in the end of the proof of theorem 2 in \cite{arma2010}).
\end{proof}

Next we recall some important results concerning the strong maximum principle and Hopf lemma. For a proof for $L^p$-viscosity solutions with unbounded coefficients, see \cite{B2016}, which in particular generalize the results for $C$-viscosity solutions in \cite{BardidaLio}.
We will refer to them simply by \textit{SMP} and \textit{Hopf} throughout the text.

\begin{teo}{$($SMP\,$)$}\label{SMP}
Let $\Omega$ be a $C^{1,1}$ domain and $u$ an $L^p$-viscosity solution of $\mathcal{L}^-[u]-du\leq 0$, $u\geq 0$ in $\Omega$, where $d\in L^p(\Omega)$. Then either $u>0$ in $\Omega$ or $u\equiv 0$ in $\Omega$.
\end{teo}

\begin{teo}{$($Hopf\,$)$}\label{Hopf}
Let $\Omega$ be a $C^{1,1}$ domain and $u$ an $L^p$-viscosity solution of $\mathcal{L}^-[u]-du\leq 0$, $u> 0$ in $\Omega$, where $d\in L^p(\Omega)$. If $u(x_0)=0$ for some $x_0\in\partial\Omega$, then $\partial_\nu u(x_0)>0$, where $\partial_\nu$ is the derivative in the direction of the interior unit normal.
\end{teo}

In \cite{B2016}, theorems \ref{SMP} and \ref{Hopf} are proved for $d\equiv 0$, but exactly the same proofs there work for any coercive operator. Moreover, since the function $u$ has a sign, they are also valid for nonproper operators, by splitting the positive and negative parts of $d$ and using $d^- u\geq 0$.

We finish the section recalling some results about pure second order operators $F(D^2 u)$, i.e. uniformly elliptic operators $F$ depending only on $X$ (so Lipschitz continuous in $X$) and satisfying $F(0)=0$. These operators will play the role of $F(0,0,0,X)$ in the approximation lemmas.

The next proposition is corollary 5.7 in \cite{CafCab}, which deals with $C^{1,\bar{\alpha}}$ interior regularity.

\begin{prop}
\label{C1,baralpha}
Let $u$ be a $C$-viscosity solution of $F(D^2 u)=0$ in $B_1$. Then $u\in C^{1,\bar{\alpha}} (\overline{B}_{1/2})$ for some universal $\bar{\alpha}\in (0,1)$ and there exists a constant $K_2$, depending on $n,\lambda$ and $\Lambda$, such that
$$\|u\|_{C^{1,\bar{\alpha}} (\overline{B}_{1/2})} \leq K_2 \, \|u\|_{L^\infty (B_1)} . $$
\end{prop}

We also need the following result about solvability of the Dirichlet problem for $F(D^2 u)$.
\begin{prop}
\label{ExisUnicF(D2u)}
Let $\Omega$ satisfies a uniform exterior cone condition, $\psi\in C(\partial\Omega)$. Then there exists a unique $C$-viscosity solution $u\in C(\overline{\Omega})$ of
\begin{align*}
\left\{
\begin{array}{rclcc}
F(D^2 u)&=& 0  &\mbox{in} &\Omega\\
u &=& \psi &\mbox{on} & \partial\Omega\, .
\end{array}
\right.
\end{align*}
\end{prop}

\begin{proof}
Uniqueness is corollary 5.4 in \cite{CafCab}. Let us recall how to obtain existence via Perron's Method, proposition II.1 in \cite{IL90} (see also \cite{Ishii89}). Surely, comparison principle holds for $F(D^2 u)$ by theorem 5.3 and corollary 3.7 in \cite{CafCab}. Further, we obtain a pair of strong sub and supersolutions $\underline{u},\, \overline{u} \in W^{2,p}_{\mathrm{loc}}(\Omega) \cap C(\overline{\Omega})$ of Pucci's equations $\mathcal{M}^+ (D^2 \overline{u}) \leq 0 \leq \mathcal{M}^- (D^2 \underline{u})$ in $\Omega$ with $\underline{u}= \overline{u}=\psi$ on $\partial\Omega$ by lemma 3.1 of \cite{CCKS}. They are $L^p$ (so $C$) viscosity
sub and supersolutions of $F(D^2 u)=0$.
\end{proof}

We use the following notation from \cite{MilSilv} and \cite{Winter},
$$B_r^\nu (x_0):=B_r(x_0)\cap \{x_n >-\nu\} ,\;\;\mathbb{T}^{\nu}_r(x_0) :=B_r (x_0)\cap \{ x_n =-\nu \},\textrm{ for } r>0 ,\,\nu >0,$$
simply $\mathbb{T}:=B_1 \cap\{x_n=0\}$ and $B_r^+:=B_r \cap \{ x_n >0 \}$, where $B_r=B_r(0)$.

\begin{prop}
 \label{C1,baralphaglobal}
Let $u\in C(\overline{B_1^\nu })$ be a $C$-viscosity solution of
\begin{align*}
\left\{
\begin{array}{rclcc}
F(D^2 u)& =& 0  &\mbox{in} & B_1^\nu \\
u &=&\psi  &\mbox{on} &\mathbb{T}_1^\nu
\end{array}
\right.
\end{align*}
such that $\psi\in C(\partial B_1^\nu)\cap C^{1,\tau}(\mathbb{T}_1^\nu)$ for some $\tau >0$. Then $u\in C^{1,\bar{\alpha}}( \overline{B_{1/2}^\nu})$, where $\bar{\alpha}=\min (\tau, \alpha_0)$ for a universal $\alpha_0$. Moreover, for a constant $K_3$, depending only on $n,\lambda,\Lambda$ and $\tau$, we have
\begin{align*}
\|u\|_{C^{1,\bar{\alpha}}( \overline{B_{1/2}^\nu})}\leq K_3 \, \{ \|u\|_{L^\infty (B_1^\nu)} + \|\psi\|_{C^{1,\tau} (\mathbb{T}_1^\nu )} \}.
\end{align*}
\end{prop}

For a proof of proposition \ref{C1,baralphaglobal} see proposition 2.2 in \cite{MilSilv}; see also remark 3.3 in \cite{Winter}.

\section{Proof of theorem \ref{C1,alpha regularity estimates geral}.}\label{proof main th}

\subsection{Local Regularity}\label{local regularity}

Fix a domain $\Omega^\prime\subset\subset\Omega$.
Consider $K_1$ and $\beta$ the pair given by the $C^\beta$ local superlinear estimate (proposition \ref{Cbetaquad}) for $\Omega^\prime $, related to the initial $n,p,\lambda, \Lambda,\mu,\|b\|_{L^p(\Omega)}$, $\mathrm{dist}(\Omega^\prime,\partial\Omega)$ and $C_0$ such that
$$\|u\|_{C^\beta(\Omega^\prime)}\leq K_1 \,\{\|u\|_{L^\infty (\Omega)} + \|f\|_{L^p(\Omega)}+\|d\|_{L^p(\Omega)}\,\omega(\|u\|_{L^\infty (\Omega)})\}.$$

Also, let $K_2$ (which we can suppose greater than 1) and $\bar{\alpha}$ be the constants of $C^{1,\bar{\alpha}}$ local estimate (proposition \ref{C1,baralpha}) associated to $n,\lambda,\Lambda$ in the ball $B_1(0)$. 

By taking $K_1$ larger and $\beta$ smaller, we can suppose $K_1\geq \widetilde{K}_1$ and $\beta\leq \widetilde{\beta}$, where $\widetilde{K}_1,\widetilde{\beta}$ is the pair of $C^\beta$ local estimate in the ball $B_1$ (or $B_{1/2}$), with respect to an equation with given constants $n,p,\lambda,\Lambda$ and bounds for the coefficients $\mu\leq 1$, $\|b\|_{L^p(B_2)}\leq 1+2K_2|B_1|^{1/p}$ and $\omega(1)\|d\|_{L^p(B_2)}\leq 1$, for all solutions in the ball $B_2$ with $\|u\|_{L^\infty(B_2)}\leq 1$ (or for all solutions in the ball $B_1$ with bounds on the coefficients in $B_1$).

\vspace{0.1cm}

The first step is to approximate our equation with one which already has the corresponding regularity and estimates that we are interested in. Denote $\|\cdot\|_{ p}=\|\cdot\|_{ L^p (B_1)}$.

\begin{lem}
\label{AproxLem}
Assume $F$ satisfies \ref{SCmu} in $B_1$, $f \in L^p (B_1)$, where $p>n$. Let $\psi \in C^\tau (\partial B_1)$ with $\|\psi\|_{C^\tau (\partial B_1)}\leq K_0$. 
Moreover, set $L(x)=A+B\cdot x$ in $B_1$, for $A\in \real$, $B\in \rn$.
Then, for every $\varepsilon>0$, there exists $\delta\in (0,1)$, $\delta=\delta (\varepsilon,n,p,\lambda,\Lambda,\tau,K_0)$, such that
\begin{align*}
\|{\bar{\beta}}_F(\cdot,0)\|_{ p}\, ,\;\|f\|_{ p}\, ,\; \mu(|B|^2 +|B|+1) \, ,\;
\|b\|_{p}(|B|+1) \, ,\;\omega (1) \|d\|_{ p}(|A|+|B|+1)\leq \delta
\end{align*}
imply that any two $L^p$-viscosity solutions $v$ and $h$ of
\begin{align*}
\left\{
\begin{array}{rclcc}
F(x,v+L(x),Dv+B,D^2v)&=& f(x)& \mbox{in} & B_1 \\
v &=& \psi \; & \mbox{on} & \partial B_1
\end{array}
\right., \;
\left\{
\begin{array}{rclcc}
F(0,0,0,D^2h)&=& 0 & \mbox{in} & B_1 \\
h &=&\psi & \textrm{on} & \partial B_1
\end{array}
\right.
\end{align*}
respectively, with $\omega (1) \|d\|_{ p}\|v\|_{\infty}\leq \delta$, satisfy $\|v-h\|_{L^\infty (B_1)}\leq \varepsilon$. 
\end{lem}

\begin{proof}
In this proof denote $\alpha_n$ as the measure of the ball $B_1(0)$.
We are going to prove that for all $\varepsilon>0$, there exists a $\delta\in (0,1)$ satisfying the above, with $\delta\leq 2^{-\frac{n}{2p}}C_n^{-\frac{1}{2}}\,\widetilde{\delta}^{1/2}$, where $\widetilde{\delta}$ is the constant from proposition \ref{wABPquad} for $\tilde{b}=b+2|B|\mu$, $C_n=4+2\alpha_n^{{1}/{p}}$. 
Assume the conclusion is not satisfied, then there exist some $\varepsilon_0>0$ and a sequence of operators $F_k$ satisfying $(SC)^{\mu_k}$ for $b_k,\, d_k\in L^p_+(B_1)$, $\mu_k\geq 0$, $\omega_k$ modulus, also $f_k\in L^p(B_1)$, $A_k\in \real$, $B_k\in \rn$, $L_k(x)=A_k+B_k\cdot x$, and $\delta_k\in(0,1)$ such that
$
\delta_k \leq 2^{-\frac{n}{2p}}C_n^{-\frac{1}{2}}  \,\widetilde{\delta}_k^{1/2}$ for all $k\in\n$,
where $\widetilde{\delta}_k$ is the number from ABP related to $\tilde{b}_k=b_k+2|B_k|\mu_k$, in addition to
\begin{align*}
\|{\bar{\beta}}_{F_k}(\cdot,0)\|_{ p},\; \|f_k\|_{p} ,\; \mu_k (|B_k|^2+|B_k|+1),\;
\|b_k\|_{ p}(|B_k|+1),\; \omega_{k}(1) \|d_k\|_{ p}(|A_k|+|B_k|+1)\leq \delta_k 
\end{align*}
with $\delta_k\rightarrow 0$, and $v_k$, $h_k\in C(\overline{B}_1)$ $L^p$-viscosity solutions of
\begin{align*}
\left\{
\begin{array}{rclc}
F_k(x,v_k+L_k(x),Dv_k+B_k,D^2v_k)&=&f_k(x) & B_1 \\
v_k &=& \psi_k  & \partial B_1
\end{array}
\right.
, \;
\left\{
\begin{array}{rclc}
F_k(0,0,0,D^2h_k)&=& 0 &  B_1 \\
h_k &=& \psi_k & \partial B_1
\end{array}
\right.
\end{align*}
where $\|\psi_k\|_{C^\tau (\partial B_1)}\leq K_0$, $\omega_k (1) \|d_k\|_{ p}\|v_k\|_{\infty}\leq \delta_k$, but
$\|v_k-h_k\|_{L^\infty (B_1)} >\varepsilon_0.$
We first claim that
\begin{align} \label{v,hbound}
\|v_k\|_{L^\infty (B_1)}\,, \;\|h_k\|_{L^\infty (B_1)}\leq C_0
\end{align}
for large $k$, where $C_0=C_0(n,p,\lambda,\Lambda, K_0)$.
Indeed, in the first place, since we have $\mathcal{M}^-(D^2h_k) \leq 0\leq \mathcal{M}^+(D^2h_k)$ in the viscosity sense, we obtain directly that
$
\|h_k\|_{L^\infty (B_1)}\leq \|\psi_k\|_{L^\infty (\partial B_1)}\leq K_0.
$
For $v_k$, we initially observe that
$$
2^{\frac{n}{p}} C_n\,  \mu_k \, \delta_k \leq 2^{\frac{n}{p}} C_n \,\delta_k^2\leq  \widetilde{\delta}_k\, , \;\;\; \textrm{for all } k\in\n.
$$
Further, $v_k$ is an $L^p$-viscosity solution of
$$\mathcal{{L}}_k^+[v_k]+\mu_k|D v_k|^2 +d_k(x) \omega_k (|v_k+L_k|)+g_k \geq f_k \geq \mathcal{{L}}_k^-[v_k]-\mu_k |Dv_k|^2-d_k(x)\omega_k(|v_k+L_k|)-g_k.$$
Here $\mathcal{{L}}_k^\pm[w]=\mathcal{M}^\pm(D^2w)\pm \widetilde{b}_k|Dw|$, and $g_k(x)=b_k(x)|B_k|+\mu_k |B_k|^2$. 
Then, applying ABP in its quadratic form in $B_1(0)$, with RHS  $d_k \,\omega_k (|v_k+L_k|)+g_k+|f_k|$, we obtain that
\begin{align*}
\|v_k\|_{L^\infty(B_1)} &\leq\| v_k\|_{L^\infty(\partial B_1)} +C_A^k  \, \{\|{f}_k\|_{p}+\|{g}_k\|_{p} +\|d_k\|_{p}\,\omega_k(1) (|A_k|+|B_k|+\|v_k\|_{L^\infty (B_1)}+1) \}.
\end{align*}
Since $\|\tilde{b}_k\|_{L^n(B_1)}\leq \alpha_n^{\frac{p-n}{np}}$ for large $k$, then the constant in ABP is uniformly bounded, say $C_A^k\leq C_A$. Using the assumptions and $C_A\,\omega_k(1)\|d_k\|_{p}\leq 1/2$ for large $k$ as in \cite{Winter}, we obtain that $\|v_k\|_{L^\infty(B_1)} \leq C_0 $, with $C_0=C_0(n,p,\lambda,\Lambda, K_0)$, proving the claim \eqref{v,hbound}.

Then, by the $C^\beta$ global estimate (proposition \ref{Cbetaquad}), there exists $\beta\in (0,1)$ such that
\begin{align*}
\|v_k\|_{C^\beta (\overline{B}_1)}\,, \;\|h_k\|_{C^\beta (\overline{B}_1)}\leq C ,
\end{align*}
where $\beta=\min {(\beta_0,\frac{\tau}{2})}$ for some $\beta_0=\beta_0 (n,p, \lambda, \Lambda)$, $C=C(n,p, \lambda, \Lambda, C_0)$. Here, $\beta$ and $C$ do not depend on $k$, since $\mu_k, \, \|\widetilde{b}_k\|_{L^p(B_1)}, \, \omega_k (1)\|d_k\|_{L^p(B_1)}, \, \|f_k\|_{L^p(B_1)}\leq 1$ for large $ k$.
Then, by the compact inclusion $C^\beta(\overline{B}_1) \subset C(\overline{B}_1)$ we have, up to subsequences, that
$$
v_k\longrightarrow v_\infty\,, \;\;\;  h_k\longrightarrow h_\infty\;\;\; \textrm{in}\; C(\overline{B}_1)\;\;\;\; \textrm{as}\;{k\rightarrow\infty},
$$
for some $v_\infty, \; h_\infty \in C(\overline{B}_1)$ with $v_\infty = h_\infty = \psi_\infty$ on $\partial B_1$. Moreover, by Arzelà-Ascoli theorem, a subsequence of $F_k(0,0,0,X)$ converges uniformly on compact sets of $\mathbb{S}^n$ to some uniformly elliptic operator $F_\infty(X)$, since $ \mathcal{M}^-_{\lambda,\Lambda}(X-Y) \leq F_k(0,0,0,X)-F_k(0,0,0,Y) \leq \mathcal{M}^+_{\lambda,\Lambda}(X-Y)$.

We claim that both $v_\infty$ and $h_\infty$ are viscosity solutions of
\begin{align*}
\left\{
\begin{array}{rclcc}
F_\infty (D^2 u)&=&0 & \mbox{in} & B_1 \\
u&=&\psi_\infty  & \mbox{on} & \partial B_1 \,.
\end{array}
\right.
\end{align*}
This implies that they are equal, by proposition \ref{ExisUnicF(D2u)},
which contradicts $\|v_\infty-h_\infty\|_{L^\infty (B_1)}\geq \varepsilon_0$.

The claim for $h_\infty$ follows by taking the uniform limits at the equation satisfied by $h_k$. On the other hand, for $v_\infty$ we apply stability (proposition \ref{Lpquad}) by noticing that, for $\varphi\in C^2(B_1)$,
\begin{align*}
F_k(x,v_k+L_k,D\varphi+B_k, D^2 \varphi)-f_k(x)-F_\infty (D^2 \varphi)= \{ F_k(x,v_k+L_k,D\varphi+B_k, D^2 \varphi)-\\ F_k(x,0,0, D^2 \varphi) \} 
+\{ F_k(x,0,0, D^2 \varphi)- F_k(0,0,0, D^2 \varphi) \}
+ \{ F_k(0,0,0, D^2 \varphi) - F_\infty (D^2 \varphi) \} -f_k(x)
\end{align*}
and that each one of the addends in braces tends to zero in $L^p$ as $k\rightarrow\infty$. Indeed, the first one in modulus is less or equal than
$\mu_k (|D\varphi |^2+2B_k|D\varphi|+|B_k|^2)+ b_k(x)( |D\varphi|+|B_k|)+\omega_k (\|v_k\|_{\infty } +|A_k|+|B_k| ) \, d_k(x)$, so its $L^p$-norm is bounded by $\mu_k \alpha_n^{1/p} \|D\varphi\|_{\infty}^2 + \| b_k\|_{p} \|D\varphi\|_{\infty} + (C_0+1)\,\omega_{k}(1) \, \| d_k\|_{p} +C_n\delta_k$; while the $L^p$-norm of the second and third are bounded by $\|{\bar{\beta}}_{F_k}(\cdot,0)\|_{p} (\|D^2 \varphi \|_{\infty} +1)$ and $\alpha_n^{{1}/{p}} \| F_k(0,0,0, D^2 \varphi) - F_\infty (D^2 \varphi) \|_{L^\infty (B_1)}$ respectively, what concludes the proof.
\end{proof}

\begin{proof}[\textit{Proof of Local Regularity Estimates in the set $\Omega^\prime$.}] The main difference from the case $\mu= 0$, in the present proof, consists of defining a slightly different scaling on the function, which allows us to have $\mu$ small in order to obtain the conditions of the approximation lemma \ref{AproxLem}. For this, we will bring forward an argument due to Wang \cite{W2}, that uses the $C^\beta$ regularity of $u$.

Set $W:=\| u \|_{L^{\infty} (\Omega)} + \|f \|_{L^p (\Omega)}+\|d\|_{L^p(\Omega)}\,\omega(\|u\|_{L^\infty (\Omega)})$, which is less or equal than $W_0$, a constant that depends on $C_0$ and $\omega (1)\|d\|_{L^p(\Omega)}$.

Suppose, to simplify the notation, that $0\in \Omega^\prime$ and set $s_0:=\min (r_0, \mathrm{dist}(0,\partial\Omega^\prime))$. Recall that this $r_0 = r_0 (\theta)$ is such that \eqref{Htheta} holds for all $r\leq \min \{ r_0, \mathrm{dist} (x_0,\partial\Omega)\}$, for all $x_0 \in \Omega$. We will see, in the sequel, how the choice of $\theta$ is done.

We start assigning some constants. Fix an $\alpha\in(0,\bar{\alpha})$ with $\alpha\leq\min (\beta,1- {n}/{p} )$. Then, choose $\gamma=\gamma(\alpha,\bar{\alpha},K_2)\in (0,\frac{1}{4}]$ such that
\begin{align} \label{gamma}
2^{2+\bar{\alpha}}K_2\, \gamma^{\bar{\alpha}}\leq \gamma^{\alpha}
\end{align}
and define
\begin{align} \label{epsilon}
\varepsilon=\varepsilon (\gamma):=K_2\, (2\gamma)^{1+\bar{\alpha}}.
\end{align}
This $\varepsilon$ provides a $\delta=\delta(\varepsilon)\in (0,1)$, the constant of the approximation lemma \ref{AproxLem} that, up to diminishing, can be supposed to satisfy
\begin{align} \label{delta}
(5+2K_2)\,\delta \leq \gamma^{\alpha}.
\end{align}
Now let $\sigma=\sigma (s_0,n,p,\alpha , \bar{\alpha},\beta, \delta,\mu ,\|b\|_{L^p (\Omega)},\omega(1)\|d\|_{L^\infty (\Omega)},K_1,K_2,C_0) \leq \frac{s_0}{2}$ such that
\begin{align} \label{sigma}
\sigma^{\min{({1-\frac{n}{p}},\beta)}} m \leq {\delta} \,\{ {32K^2(K_2+K+1)|B_1|^{1/p}}\}^{-1}
\end{align}
where $m:=\max{ \{ 1, \|{b}\|_{L^p (\Omega)}, \omega(1)\|{d}\|_{L^\infty (\Omega)},\mu (1+2^\beta K_1) W_0 \} }$. Consider the constant $K (\gamma, \alpha, K_2)$ defined as $K={K_2}\,{{\gamma^{-\alpha}}(1-\gamma^\alpha)^{-1}} +{K_2}\,{{\gamma^{-1-\alpha}}(1-\gamma^{1+\alpha})^{-1}} $ which is greater than $K_2\geq 1$. Hence, in particular, $\overline{B}_{2\sigma} (0)\subset \Omega^\prime$ and we can define
\begin{align*}
N=N_\sigma (0):= \sigma W + \sup_{x\in B_2} |u(\sigma x) - u(0)|.
\end{align*}
By construction and $C^\beta$ local quadratic estimate, $N$ is uniformly bounded by
\begin{align}\label{N}
\sigma W \leq N \leq (\sigma + 2^\beta K_1 \sigma ^\beta)W
\leq (1+2^\beta K_1) W_0\,\sigma ^\beta .
\end{align}

\vspace{1.1mm}
\begin{claim}\label{claim C1,alpha local 1a mud.var.} 
$\widetilde{u}(x):=\frac{1}{N} \{ u(\sigma x)-u(0) \}$
is an $L^p$-viscosity solution of $\widetilde{F}[\,\widetilde{u}\,]=\widetilde{f}(x)$ in $B_2$, where
$$\widetilde{F}(x,r,p,X):= \frac{\sigma^2}{N} F\left( \sigma x,Nr+u(0),\frac{N}{\sigma}p, \frac{N}{\sigma^2} X \right)-\frac{\sigma^2}{N} F\left( \sigma x,u(0),0, 0 \right)$$
and $\widetilde{f}:=\widetilde{f}_1+\widetilde{f}_2$ for
\begin{align*}
\widetilde{f}_1(x):={\sigma^2} f(\sigma x)/N, \quad
\widetilde{f}_2(x):= -
{\sigma^2} F\left( \sigma x,u(0),0, 0 \right)/N,
\end{align*}
with $\widetilde{F}$ satisfying $(\widetilde{SC})^{\widetilde{\mu}}$ for $\widetilde{b}(x):= \sigma b(\sigma x)$, $\widetilde{\mu}:=N\mu$,
$\widetilde{d}(x):= \sigma^2 d(\sigma x)$ and $\widetilde{\omega}(r):=\omega(Nr)/N $.
\end{claim}

\begin{proof}
Let $\varepsilon >0$ and $\widetilde{\varphi}\in W^{2,p}_{\textrm{loc}}(B_2)$ such that $\widetilde{u}-\widetilde{\varphi}$ has a minimum (maximum) at $x_0\in B_2$.
Define $\varphi (x):={N} \widetilde{\varphi} (x/\sigma)+u(0)$ in $B_{2\sigma} (0)$ and notice that $u-\varphi$ has a minimum (maximum) at $\sigma x_0 \in B_{2\sigma}$. Since $u$ is an $L^p$-viscosity solution on $B_{2\sigma}$, for this $\varepsilon >0$ there exists $r>0$ such that
\begin{align*}
F(\sigma x, u(\sigma x), D\varphi (\sigma x),D^2\varphi (\sigma x)) \leq (\geq)\, f(\sigma x)+(-)\,{N\varepsilon}/{\sigma^2} \quad \textrm{a.e. in }B_{r}(x_0),
\end{align*}
which is equivalent to
\begin{align*}
\frac{\sigma^2}{N}&F\left(\sigma x, N \widetilde{u}(x)+u(0),\frac{N}{\sigma}D \widetilde{\varphi}(x),\frac{N}{\sigma^2}D^2\widetilde{\varphi} (x)\right) \leq (\geq)\,\frac{\sigma^2}{N} f(\sigma x)+ (-)\,\varepsilon \quad \textrm{a.e. in }B_{r}(x_0).
\end{align*}
Adding $-{\sigma^2} F\left( \sigma x,u(0),0, 0 \right)/N$ in both sides, we have
\begin{align*}
\widetilde{F}(x, \widetilde{u}(x),D \widetilde{\varphi}(x),D^2\widetilde{\varphi}(x)) \leq (\geq)\, \widetilde{f}( x)+ (-)\,\varepsilon \quad \textrm{a.e. in }B_{r}(x_0).
\end{align*}
Furthermore, $\widetilde{F}(x,0,0,0)= 0\,$ for $x\in B_2$ and for all $r\in \real$, $p\in \rn$, $X\in \mathbb{S}^n$, we have
\begin{align*}
\widetilde{F} &(x,r,p,X)-\widetilde{F}(x,s,q,Y)\\
&= \frac{\sigma^2}{N} \left\{ F\left( \sigma x, Nr+u(0),\frac{N}{\sigma}p,\frac{N}{\sigma^2} X\right)-  F\left( \sigma x, Ns+u(0),\frac{N}{\sigma}q,\frac{N}{\sigma^2} Y\right)  \right\}  \\
&\leq \mathcal{M}^+_{\lambda,\Lambda} (X-Y)+\sigma b(\sigma x) \,|p-q|+N\mu |p-q|(|p|+|q|)+\sigma^2 d(\sigma x) \,{\omega(N|r-s|)}/{N}\\
&=\mathcal{M}^+_{\lambda,\Lambda}(X-Y)+\widetilde{b}(x)|p-q|+\widetilde{\mu} |p-q|(|p|+|q|) +\widetilde{d}(x)\,{\widetilde{\omega} }(|r-s|).
\end{align*}
The estimate from below in $(\widetilde{SC})^{\widetilde{\mu}}$ is analogous.
\qedhere{\textit{Claim \ref{claim C1,alpha local 1a mud.var.}.}}
\end{proof}

Notice that, with this definition and the choice of $\sigma$ in \eqref{sigma}, we have
\begin{itemize}
\item $\|\widetilde{u}\|_{L^\infty (B_2)} \leq 1$ since $N\geq \sup_{B_2} |u(\sigma x)-u(0)|$;

\item $\|\widetilde{f}_1\|_{L^p (B_2)}=\frac{\sigma^{2-\frac{n}{p}}}{N}\|{f}\|_{L^p (B_{2\sigma})}\leq \sigma^{1-\frac{n}{p}} \frac{\|{f}\|_{L^p (\Omega)}}{W} \leq \frac{\delta}{16}$;

\item $\|\widetilde{f}_2\|_{L^p (B_2)}\leq \frac{\sigma^{2-\frac{n}{p}}}{N} \omega (|u(0)|) \,\|d\|_{L^p (B_{2\sigma})}$ $\leq \sigma^{1-\frac{n}{p}}\frac {\omega (\|u\|_\infty)  \|d\|_{L^p (\Omega)} }{W}  \leq  \frac{\delta}{16}$\,; thus $\|\widetilde{f}\|_{L^p (B_2)}\leq \frac{\delta}{8}$;

\item $\widetilde{\mu}=N\mu\leq (1+2^\beta K_1)W_0 \, \mu \, \sigma^\beta \leq \frac{\delta}{8K^2|B_1|^{1/p} }$;

\item $\|\widetilde{b}\|_{L^p (B_2)}= \sigma^{1-\frac{n}{p}} \|{b}\|_{L^p (B_{2\sigma})} \leq \frac{\delta}{16K} $;

\item $\widetilde{\omega}(1) \|\widetilde{d}\|_{L^p (B_2)}={\sigma^{2-\frac{n}{p}}} \frac{\omega(N)}{N} \|{d}\|_{L^p (B_{2\sigma})}\leq  \sigma^{2-\frac{n}{p}}  \omega(1) \|{d}\|_{L^p (\Omega)} \leq \frac{\delta}{32(K_2+K+1)}$ from the hypothesis $\omega(r)\leq \omega (1)r$ for all $r\geq 0$;

\item $\|{\bar{\beta}}_{\widetilde{F}}(\cdot,0)\|_{L^p (B_1)} \leq \delta/4 $, by choosing $\theta=\delta/8$. Indeed,
\end{itemize}
\begin{align}\label{mudandobetabar}
\bar{\beta}_{\widetilde{F}}(x,x_0)&\leq \frac{\sigma^2}{N}\sup_{X\in \mathbb{S}^n} \frac{| F(\sigma x,u(0),0,\frac{N}{\sigma^2}X)-F(\sigma x,0,0,\frac{N}{\sigma^2}X)|}{\|X\|+1} \nonumber\\
&\quad +\sup _{X\in \mathbb{S}^n} \frac{| F(\sigma x,0,0,\frac{N}{\sigma^2}X)-F(\sigma x_0,0,0,\frac{N}{\sigma^2}X)|}{\frac{N}{\sigma^2}(\|X\|+1)} \nonumber\\
&\quad +\frac{\sigma^2}{N}\sup_{X\in \mathbb{S}^n} \frac{| F(\sigma x_0,0,0,\frac{N}{\sigma^2}X)-F(\sigma x_0,u(0),0,\frac{N}{\sigma^2}X)|}{\|X\|+1} \nonumber\\
&\quad + \frac{\sigma^2}{N}\sup_{X\in \mathbb{S}^n} \frac{| F(\sigma x,u(0),0,0)|+|F(\sigma x_0,u(0),0,0)|}{\|X\|+1}\nonumber\\
&\leq \frac{2\sigma^2}{N} \{d(\sigma x)+d(\sigma x_0)\} \,\omega(|u(0)|)\sup_{X\in \mathbb{S}^n} ({\|X\|+1} )^{-1}+\beta_F(\sigma x,\sigma x_0)
\end{align}
and therefore,
\begin{align*}
\|\bar{\beta}_{\widetilde{F}}(\cdot,0)\|_{L^p (B_1)}&\leq 4 \sigma |B_1|^{\frac{
1}{p}}\, \frac{\omega (\|u\|_{L^\infty(\Omega)})\|d\|_{L^\infty(\Omega)}}{W} + \left( \frac{1}{\sigma^n} \int_{B_\sigma(0)} \beta_F(y,0)^p \mathrm{d}y \right)^{\frac{1}{p}}\\
&\leq \delta/8 +\theta\,=\, \delta/4.
\end{align*}

In particular $\widetilde{F}, \widetilde{u}, \widetilde{\mu}, \widetilde{b}, \widetilde{d}, \widetilde{\omega}, A=0, B=0$ satisfy lemma \ref{AproxLem} hypotheses. Further, if we show $\|\widetilde{u}\|_{C^{1,\alpha}(\overline{B}_1)} \leq C$, we will obtain
$
\| {u(\sigma x)-u(0)} \|_{C^{1,\alpha}(\overline{B}_1)} \leq CN \leq (1+2^\beta K_1)C W
$
by \eqref{N}, then
$$
\| {u} \|_{C^{1,\alpha}(\overline{B}_\sigma)}\leq C\, \{ \| u \|_{L^{\infty} (\Omega)} + \|f \|_{L^p (\Omega)}\}  ,
$$
where the constant depends on $\sigma$; the local estimate following by a covering argument.

\begin{rmk}\label{Remark qualquer modulo} In the case we have an arbitrary modulus of continuity, we define
$
N= \sigma \max \{W,1\} \\ + \sup_{x\in B_2} |u(\sigma x) - u(0)|,
$
which by construction and $C^\beta$ local superlinear estimate,
$$
\sigma \leq N \leq (\sigma + 2^\beta K_1 \sigma ^\beta)\max\{W,1\}
\leq (1+2^\beta K_1) W_0\,\sigma ^\beta \leq 1.
$$
Then we have
$\widetilde{\omega}(1) \|\widetilde{d}\|_{L^p (B_2)}=\frac{\sigma^{2-\frac{n}{p}}}{N}  \omega(N) \|{d}\|_{L^p (B_{2\sigma})}\leq  \sigma^{1-\frac{n}{p}}  \omega(1) \|{d}\|_{L^p (\Omega)} $.

Moreover, we can consider the smallness assumption in terms of $(\overline{H})_\theta$, with $\bar{\beta}$ instead of $\beta$ in \eqref{mudandobetabar}. In fact, in this case we use $N/{\sigma^2}\geq 1$. In the end, we obtain that the original function $u$ is such that $\|u\|_{C^{1,\alpha}(\overline{\Omega})}$ is bounded by $C\max\{W,1\}\leq C (W+1)$, in place of $CW$.
\end{rmk}

Notice that the only place we had to use the dependence on the bound $C_0$ is to measure the smallness of $\mu$. Thus, if $\mu=0$, the final constant does not depend on $W_0$, neither on $C_0$.
\begin{rmk}\label{remark mu=0}
Still for $\mu=0$, if we split our analysis in two cases $($as usual for linear growth in the gradient, see for instance \cite{Bbook}$\,)$, then we can obtain the conditions in terms of $(\overline{H}_\theta)$. Indeed, set $N:=W$. If $N\leq 1$, we just define $\widetilde{u}=u (\sigma x)$ and use that each of the addends in $W$ is less or equal than $1$, and
$\widetilde{\omega}(1) \|\widetilde{d}\|_{L^p (B_2)}={\sigma^{2-\frac{n}{p}}} {\omega(N)}\|{d}\|_{L^p (B_{2\sigma})}\leq  \sigma^{2-\frac{n}{p}}  \omega(1) \|{d}\|_{L^p (\Omega)}$; this yields the estimate
$\|u\|_{C^{1,\alpha}}\leq C\leq C (N+1)$. 
If $N\geq 1$, $\widetilde{u}$ is as in claim \ref{claim C1,alpha local 1a mud.var.}, and using $N/{\sigma^2}\geq 1$ we can replace $\beta$ by $\bar{\beta}$ in \eqref{mudandobetabar}; for the estimate in $\widetilde{\omega}(1) \|\widetilde{d}\|_{L^p (B_2)}$ we only need $\omega(r)\leq \omega (1)r$ for $r\geq 1$.
\end{rmk}\medskip

With these rescalings in mind, we write $F,u,M, \mu,b,d,\omega$ as the shorthand notation for $\widetilde{F},\widetilde{u}, \widetilde{\mu},\widetilde{b},\widetilde{d},\widetilde{\omega}$.
Now we can proceed with Caffarelli's iterations as in \cite{Caf89}, \cite{CafCab}, \cite{Swiech}, which consists of finding a sequence of linear functions $l_k (x):=a_k +b_k \cdot x$ such that
\vspace{0.2cm}

$(i)_k \;\;\| u-l_k \|_{L^\infty (B_{r_k})} \leq r_k^{1+\alpha}$
\vspace{0.2cm}

$(ii)_k \;\;|a_k-a_{k-1}| \leq K_2 \,r_{k-1}^{1+\alpha}\;,\;\; |b_k - b_{k-1}|\leq K_2 \,r_{k-1}^\alpha$
\vspace{0.2cm}

$(iii)_k \;\;|(u-l_k)(r_k x)-(u-l_k)(r_k y)|\leq (1+3K_1) \, r_k^{1+\alpha} |x-y|^\beta \;\; \textrm{for all } x,y\in B_1 $\\
\vspace{0.2cm}
for $r_k=\gamma^k$ for some $\gamma\in(0,1)$, for all $ k\geq 0$, with the convention that $l_{-1}\equiv 0$.

Observe that this proves the result. Indeed, $b_k = b_0 + (b_1-b_0) +  \ldots + (b_k-b_{k-1})$  converges to some $b$, since $\sum_{k=0}^\infty |b_{k} - b_{k-1}| \leq K_2 \sum_{k=0}^\infty (\gamma^{\alpha} )^{k-1}  < \infty$; also
$|b_k - b| \leq \sum_{l=k}^\infty |b_{l+1} - b_l| \leq K_2 \sum_{l=k}^\infty \gamma^{\alpha l} = K_2 \frac{\gamma^{\alpha k}}{1-\gamma^\alpha}.$
Similarly, $|a_k - a|\leq K_2\frac{\gamma^{k(1+\alpha)}}{1-\gamma^{1+\alpha}}$ and $a_k$ converges to some $a$.\medskip

Next, for each $x\in B_1$, there exists $k\geq 0$ such that $r_{k+1} < |x|\le r_k$. Then, $|u(x) - a_k - b_k \cdot x|=|u(x)-l_k(x)| \leq r_k^{1+\alpha}$,
since $x\in \overline{B}_{r_k}$, thus
\begin{align*}
|u(x) - a - b\cdot x| &\leq |u(x) - a_k - b_k \cdot x| + |a_k - a| + |b_k - b| \,|x| \\
&\leq r_k^{1+\alpha} + K_2 \frac{r_k^{1+\alpha}}{1-\gamma^{1+\alpha}} + K_2 \frac{r_k^{\alpha}}{1-\gamma^\alpha} r_{k} \\
&=  \left\{ 1+\frac{K_2 }{1-\gamma^{1+\alpha}} + \frac{K_2 }{1-\gamma^\alpha}  \right\} \frac{1}{\gamma^{1+\alpha}} \, r_{k+1}^{1+\alpha} \leq C_\gamma\, |x|^{1+\alpha}.
\end{align*}
By definition of a differentiable function,  $a=u(0)\, ,\; b=D u(0)$ and we will have obtained $| u(x) - u(0) - D u(0)\cdot x | \leq C|x|^{1+\alpha}$ and $|D u(0)| \leq C$.

Notice that there was nothing special in doing the initial argument around $0$, which we had supposed in the beginning of the proof, belonging to $\Omega^\prime$. Actually, by replacing it by any $x_0\in\Omega$ and setting the corresponding $s_0=\min \{r_0,\mathrm{dist}(x_0,\partial\Omega^\prime)\}$, we define $N=N_\sigma (x_0)$ by changing $0$ by $x_0$ in there. With this, we show that our initial function $u$ is differentiable at $x_0$ with
$$
\,| u(x)-u(x_0)-Du(x_0)\cdot(x-x_0) | \leq CW|x-x_0|^{1+\alpha},\quad |Du(x_0)| \leq CW
$$
which implies\footnote{This is just a property of functions. See, for example, a simple proof done by Sirakov in \cite{Bbook}, or \cite{tese}.} that $Du\in C^\alpha (\overline{B}_{\sigma})$ and $\|u\|_{C^{1,\alpha} (\overline{B}_{\sigma})} \leq CW$.
Thus, for the complete local estimate, we just take finitely many such points in order to cover $\Omega^\prime$.

We stress that $(i)_k$ and $(ii)_k$ are completely enough to imply the result, as above, while $(iii)_k$ is an auxiliary tool to get them.
So, let us prove $(i)_k-(iii)_k$ by induction on $k$.

For $k=0$ we set $a_0 = b_0 = 0$. Recall that $\beta$ and $K_1$ are the constants from the $C^\beta$ superlinear local estimate in $B_1$ such that
$
\|u\|_{C^\beta(\overline{B}_1)}\leq \widetilde{K}_1 (1+\delta+1)\leq 3K_1,
$
which implies $(iii)_0$. Obviously $(i)_0$ and $(ii)_0$ are satisfied too.\vspace{0.1cm}

Notice that
$
|b_k|\leq  \sum_{l=0}^k |b_{l} - b_{l-1}| \leq \frac{K_2}{\gamma^\alpha} \sum_{k=0}^\infty \gamma^{\alpha k}=\frac{K_2}{{\gamma^\alpha}(1-\gamma^\alpha)}\leq K
$
and also, for all $x\in B_1$,
$
|l_k(x)|\leq |a_k|+|b_k|| x|\leq \frac{K_2}{\gamma^{(1+\alpha)}} \sum_{k=0}^\infty \gamma^{(1+\alpha) k} + \frac{K_2}{\gamma^\alpha}\sum_{k=0}^\infty \gamma^{\alpha k}= K.
$ \vspace{0.1cm}

As the induction step, we suppose the items $(i)_k-(iii)_k$ valid in order to construct $a_{k+1}$ and $b_{k+1}$ for which $(i)_{k+1}-(iii)_{k+1}$ hold. Define
$$
v(x)=v_k({x}) := \frac{(u-l_k)(r_k {x})}{r_k^{1+\alpha}}=\frac{u(r_k{x}) - a_k - b_k\cdot x r_k}{r_k^{1+\alpha}}\, , \;\textrm{ for all }\, {x}\in B_2 .
$$
Note that $(i)_k$ says precisely that $ |v({x})| \leq 1$ for all $x\in B_1.$ Further, from this and $(iii)_k$ we get
$$
\|v\|_{C^\beta (\overline{B}_1)}=\|v\|_{L^\infty ({B}_1)}+ \sup_{\substack{ x,y \in B_1 \\ x \neq y }} \frac{|v(x)-v(y)|}{|x-y|^\beta}  \leq 2+3K_1=:K_0.
$$

\begin{claim} \label{claim C1,alpha local 2a mud.var.}
$v$ is an $L^p$-viscosity solution of
$F_k [v]=f_k (x) $ in $B_2$,
for $f_k:=f_k^1+f_k^2$ with
$
f_k^1(x):=r_k^{1-\alpha} f(r_k x); \;
f_k^2(x):=- r_k^{1-\alpha} F(r_k x,l_k(r_k x),b_k, 0)
$
and $F_k$ satisfying $(SC)_{F_k}^{\mu_{F_k}}$, where
\begin{align*}
F_k(x,s,p,X):=r_k^{1-\alpha} F(r_k x,r_k^{1+\alpha} s+l_k(r_kx), r_k^\alpha p+b_k, r_k^{\alpha -1}X)- r_k^{1-\alpha} F(r_k x,l_k(r_k x),b_k, 0),
\end{align*}
$b_{F_k}(x):=r_k b(r_k x)+2r_k \mu K$, $\mu_{F_k}:=r_k^{1+\alpha} \mu$, $d_{F_k}(x):=r_k^{2} d(r_k x)$ and $\omega_{F_k}(s):=r_k^{-1-\alpha} \omega(r_k^{1+\alpha}s)$.
\end{claim}

\begin{proof}
Let $\varepsilon >0$ and $\psi\in W^{2,p}_{\textrm{loc}}(B_2)$ such that $v-\psi$ has a minimum (maximum) at $x_0$.
Define $\varphi (x):=r_k^{1+\alpha}\psi (x/r_k)+l_k(x)$ for all $x\in B_{2r_k}$; then $u-\psi$ has a minimum (maximum) at $r_k x_0$.  Since $u$ is an $L^p$-viscosity solution in $B_{2r_k}(0)$, there exists $r\in (0,2)$ such that
\begin{align*}
F(r_kx,u(r_kx),D\varphi (r_k x), D^2 \varphi (r_k x))\leq (\geq) \,f(r_kx) +(-)\, r_k^{\alpha-1} \varepsilon\;\; \textrm{ a.e. in }B_r(x_0).
\end{align*}
Using that $D\psi (x)=r_k^{-\alpha} \{D\varphi(r_k x)-b_k\}$ and $D^2\psi (x)=r_k^{1-\alpha}D^2\varphi (r_kx)$ a.e., we get
\begin{align*}
r_k^{1-\alpha}&F(r_kx,r_k^{1+\alpha}v(x)+l_k(r_kx),r_k^\alpha D\psi (x)+b_k, r_k^{\alpha-1}D^2 \psi (x))
\leq (\geq)\, r_k^{1-\alpha}f(r_kx)  +(-)\,  \varepsilon
\end{align*}
a.e. in $B_r(x_0)$. Adding $- r_k^{1-\alpha} F(r_k x,l_k(r_k x),b_k, 0)$ in both sides we obtain
$$
F_k(x,v(x),D\psi, D^2\psi)\leq (\geq )\, f_k(x)+(-) \,\varepsilon\quad\textrm{a.e. in }B_r(x_0).
$$

Moreover, $F_k$ satisfies $(SC)_{F_k}^{\mu_{F_k}}$, since $F_k(x,0,0,0)=0$ for $x\in B_2$ and
\begin{align*}
&F_k(x,r,p,X)-F_k(x,s,q,Y)= r_k^{1-\alpha}\{F(r_kx,r_k^{1+\alpha}r+l_k(r_kx),r_k^\alpha p+b_k,r_k^{\alpha-1}X)\\
&\qquad-F(r_kx, r_k^{1+\alpha}s+l_k(r_kx),r_k^\alpha q+b_k,r_k^{\alpha-1}Y)\} \\&
\leq \mathcal{M}^+_{\lambda,\Lambda} (X-Y)+ r_k b(r_kx)|p-q| +r_k\mu|p-q|\{r_k^\alpha (|p|+|q|)+b_k\} + r_k^{1-\alpha} d(r_kx) \omega(r_k^{1+\alpha}|r-s|)\\
&= \mathcal{M}^+_{\lambda,\Lambda}(X-Y)+ b_{F_k}(x)|p-q| +\mu_{F_k}|p-q|(|p|+|q|)+ d_{F_k}(x)\omega_{F_k}(|r-s|)
\end{align*}
and the left hand side is completely analogous.
\qedhere{\textit{Claim \ref{claim C1,alpha local 2a mud.var.}.} }
\end{proof}

Notice that $F_k,v,\mu_{F_k},b_{F_k},d_{F_k},\omega_{F_k}, A=0, B=0$ satisfy the hypotheses of lemma \ref{AproxLem}, since 
\begin{align*}
&\|b_{F_k}\|_{L^p (B_1)}\leq r_k^{1-\frac{n}{p}} \|b\|_{L^p (B_{r_k})}+2\mu K |B_1|^{1/p} \leq \delta;\;
\|f_k^1\|_{L^p(B_1)}\leq r_k^{1-\frac{n}{p}-\alpha} \|f\|_{L^p(B_{r_k})}\leq  \frac{\delta}{2}\, ;\\
&\|f_k^2\|_{L^p(B_1)}\leq  r_k^{1-\frac{n}{p}-\alpha} \{\|b\|_{L^p(B_{r_k})}|b_k|+ (K+1) \omega (1) \|d\|_{L^p(B_{r_k})} \}+r_k^{1-\alpha}\mu |b_k|^2  |B_1|^{\frac{1}{p}}\leq  \frac{\delta}{2};\\
&\omega_{{F_k}}(1)\|d_{F_k}\|_{L^p(B_1)}= r_k^{1-\frac{n}{p}-\alpha} \omega (r_k^{1+\alpha})\|d\|_{L^p(B_{r_k})}\leq r_k^{1-\frac{n}{p}-\alpha} \omega (1)\|d\|_{L^p(B_1)} \leq \delta ; \;\|v\|_\infty \leq 1;
\end{align*}
(recall tilde notations from claim \ref{claim C1,alpha local 1a mud.var.}), and up to defining $b$ in a zero measure set,
\begin{align*}
\bar{\beta}_{F_k}(x,x_0) &\leq r_k^{1-\alpha}\sup_{X\in \mathbb{S}^n} \frac{| F(r_k x,l_k(r_k x),b_k,r_k^{\alpha-1}X)-F(r_k x,0,0,r_k^{\alpha-1}X)|}{\|X\|+1} \\
&\qquad +\sup _{X\in \mathbb{S}^n} \frac{| F(r_k x,0,0,r_k^{\alpha-1} X)-F(r_k x_0,0,0,r_k^{\alpha-1} X)|}{r_k^{\alpha-1}(\|X\|+1)} \\
&\qquad +r_k^{1-\alpha}\sup_{X\in \mathbb{S}^n} \frac{| F(r_k x_0,0,0,r_k^{\alpha-1}X)-F(r_k x_0,l_k(r_k x_0),b_k,r_k^{\alpha-1}X)|}{\|X\|+1} \\
&\qquad + r_k^{1-\alpha}\sup_{X\in \mathbb{S}^n} \frac{| F(r_k x,l_k(r_k x),b_k,0)|+|F(r_k x_0,l_k(r_k x_0),b_k,0)|}{\|X\|+1}\\
&\leq {2r_k^{1-\alpha}}\{ (d(r_k x)+d(r_k x_0))\,\omega(\|l_k(r_kx)\|_{L^\infty(\Omega)})+ (\,b(r_kx)+b(r_kx_0)\,)|b_k|+ \mu |b_k|^2  \} \\
&\qquad \sup_{X\in \mathbb{S}^n} ({\|X\|+1} )^{-1}+ \bar{\beta}_F(r_k x,r_k x_0)
\end{align*}
since $r_k^{\alpha-1}\geq 1$; then if $b$ is bounded,
\begin{align*}
\|\bar{\beta}_{F_k}(\cdot,0)\|_{L^p (B_1)}&\leq 4 r_k^{1-\alpha} |B_1|^{\frac{1}{p}} (K+1) \,\omega(1)\|d\|_{L^\infty(B_{r_k})} + 4 Kr_k^{1-\alpha}|B_1|^{\frac{1}{p}}\|b\|_{L^\infty (B_{r_k})} +2 \mu K^2 |B_1|^{\frac{1}{p}}\\
& + \|\bar{\beta}_{F}(\cdot,0)\|_{L^p (B_{r_k})}\leq  \delta .
\end{align*}
In particular, this gives an alternative proof of $C^{1,\alpha}$ results in \cite{Swiech} in the case $\mu=0$.
\smallskip

On the other hand, without boundedness assumption on $b$, we note that it also follows from the proof of claim \ref{claim C1,alpha local 2a mud.var.} that $v$ is an $L^p$-viscosity solution of 
\begin{center}
$F_k(x,v+L_k (x), Dv+B_k, D^2 v) = f_k^1(x)$\, in $B_1$, 
\end{center}
where $F_k$ is now defined as
$F_k(x,s,p,X)=r_k^{1-\alpha} F(r_kx,r_k^{1+\alpha}s,r_k^\alpha p,r_k^{\alpha-1} X)$, for $L_k (x)=A_k +B_k \cdot x$, $A_k=r_k^{-1-\alpha} a_k $ and $B_k=r_k^{-\alpha}b_k$, which satisfies $(SC)^{\mu_{F_k}}_{F_{k}}$ for $b_{F_k}(x)=r_k b(r_kx)$, 
but $\mu_{F_k}$, $d_{F_k}$, $\omega_{F_k}$, $f_k^1$ remaining as in claim \ref{claim C1,alpha local 2a mud.var.}.
Observe that we trivially have
$\|\bar{\beta}_{F_k} (\cdot,0)\|_{L^p(B_1)}\leq  \delta$ for such $F_k$. 
Furthermore,
$|B_k|\|b_{F_k}\|_{L^p (B_1)}$ $=r_k^{1-\alpha-\frac{n}{p}}\|b\|_{L^p(B_{r_k})} |b_k|\leq K\|b\|_{L^p(B_1)}\leq \delta$;
$\mu_{F_k}|B_k| (|B_k|+1)=(r_k^{1-\alpha} |b_k|^2 +r_k |b_k|)\mu \leq K(K+1)\mu \leq \delta$; and we finally get
$\omega_{F_k}(1)\|d_{F_k}\|_{L^p (B_1)}(|A_k|+|B_k|)\leq r_k^{1-\alpha - \frac{n}{p}} ( |a_k| + r_k|b_k|)\omega(1)\|d\|_{L^p(B_{r_k})}$ $\leq 2K\omega(1)\|d\|_{L^p(B_1)}\leq \delta$ if $\omega (r)\leq \omega (1) r$ for $r\geq 0$.
Thus, such $F_k,v,\mu_{F_k},b_{F_k},d_{F_k}, \omega_{F_k}, A_k, B_k$  also satisfy lemma \ref{AproxLem} hypotheses if $\omega$ is a Lipschitz modulus.

\medskip

In any case, let $h=h_k\in C(\overline{B}_1)$ be the $C$-viscosity solution of
\begin{align*}
\left\{
\begin{array}{rclcc}
F_k (0,0,0,D^2 h) &=& 0 &\mbox{in} & B_1   \\
h &=& v &\mbox{on} & \partial B_1\,.
\end{array}
\right.
\end{align*}
By ABP we have $\|h\|_{L^\infty (B_1)}\leq \|h\|_{L^\infty (\partial B_1)}\leq 1$ and by the $C^{1,\bar{\alpha}}$ local estimate (proposition \ref{C1,baralpha}), $\|h\|_{C^{1,\bar{\alpha}} (\overline{B}_{1/2})}\leq K_2\,\|h\|_{L^\infty (B_1)} \leq K_2$.
Hence, by lemma \ref{AproxLem} applied to $F_k,v,\mu_{F_k},b_{F_k}, d_{F_k},\omega_{F_k}$, $\psi:=v\mid_{\partial B_1},\tau:=\beta$, $K_0 $ and $h$ we get, for $\varepsilon$ given in \eqref{epsilon}, that
$
\|v-h\|_{L^\infty (B_1)}\leq \varepsilon.
$

\smallskip

Define $\overline{l}(x)=\overline{l}_k(x):=h(0)+Dh(0)\cdot x$ in $B_1$, then,
\begin{align} \label{wlinfty local}
\|v-\overline{l}\|_{L^\infty (B_{2\gamma})}\leq \gamma^{1+\alpha}.
\end{align}
In fact, by the choice of $\gamma\leq \frac{1}{4}$ in \eqref{gamma}, we have for all $x\in B_{2\gamma}(0)$ that
\begin{align*}
|v(x)-\overline{l}(x)|&\leq  |v(x)-h(x)| + |h(x) - h(0) - D h(0)\cdot x|\\
&\leq K_2 \,(2\gamma)^{1+\bar{\alpha}}+ K_2|x|^{1+\bar{\alpha}}\leq 2 K_2 \,(2\gamma)^{1+\bar{\alpha}}\leq \gamma^{1+\alpha}.
\end{align*}

However, inequality \eqref{wlinfty local} and the definition of $v$ imply
\begin{align*}
|u(r_k {x}) -l_k(r_kx) - r_k^{1+\alpha} h(0) - r_k^{1+\alpha} D h(0) \cdot  x |\leq r_k^{1+\alpha} \gamma^{1+\alpha} = r_{k+1}^{1+\alpha}\, \; \textrm{ for all } x \in B_{2\gamma}\,,
\end{align*}
which is equivalent to
\begin{align*}
|u(y) -l_{k+1}(y)|\leq r_k^{1+\alpha} \gamma^{1+\alpha} = r_{k+1}^{1+\alpha}\, \; \textrm{ for all } y=r_kx \in B_{2\gamma r_k}= B_{2r_{k+1}}\, ,
\end{align*}
where $l_{k+1}(y):=l_k(y)+r_k^{1+\alpha} h(0) + r_k^{\alpha} D h(0) \cdot  y \,$. Then, we define
$$
a_{k+1} := a_k + h(0)\, r_k^{1+\alpha} , \;\;\;
b_{k+1} := b_k + D h(0)\, r_k^{\alpha}
$$
obtaining $(i)_{k+1}$. Further, $|a_{k+1} - a_k|\leq K_2 \, r_k^{1+\alpha}$, $|b_{k+1} - b_k|\leq K_2 \, r_k^{\alpha}$, which is $(ii)_{k+1}$. To finish we observe that, in order to prove $(iii)_{k+1}$, it is enough to show
\begin{align}\label{iii}
\|v-\overline{l}\|_{C^\beta (\overline{B}_\gamma)}\leq (1+2K_1) \,\gamma^{1+\alpha-\beta}.
\end{align}
Indeed, if $x,y\in B_1$ and \eqref{iii} is true, then
\begin{align*}
&|(v-\overline{l})(\gamma x)-(v-\overline{l})(\gamma y)|\leq (1+2K_1)
\gamma^{1+\alpha-\beta} |\gamma x - \gamma y|^\beta \\
& \Leftrightarrow |(u-l_k)(\gamma r_{k}x)-(u-l_k)(\gamma r_{k}y)-r_k^{\alpha} Dh(0)\cdot (x-y)\gamma r_k| \leq (1+2K_1) \gamma^{1+\alpha} r_{k}^{1+\alpha} |x-y|^\beta \\
& \Leftrightarrow |(u-l_{k+1})(r_{k+1}x)-(u-l_{k+1})(r_{k+1}y)| \leq (1+2K_1) \, r_{k+1}^{1+\alpha} |x-y|^\beta.
\end{align*}

Now, we obtain \eqref{iii} applying the local quadratic $C^\beta$ estimate (proposition \ref{Cbetaquad}) to the function $w:=v-\overline{l}$, which is an $L^p$-viscosity solution in $B_2$ of the inequalities
\begin{align} \label{Cbetafinal}
\mathcal{L}_{k}^- [w]-\mu_{F_k}|Dw|^2\leq g_k(x) ,\;\;
\mathcal{L}_{k}^+ [w]+\mu_{F_k} |Dw|^2 \geq - g_k(x),
\end{align}
where $g_k:=g_k^1+g_k^2\,$, for $g_k^1(x):=|f_k(x)-F_k(x,\overline{l}(x),Dh(0),0)|$ and $g_k^2 (x):=d_{F_k}(x)\omega_{F_k}(|w|)$, with $\mathcal{L}^\pm_k[u]:=\mathcal{M}^\pm_{\lambda,\Lambda}(D^2 u)\pm  (b_{F_k}+2K_2\,\mu_{F_k})|Du|$.
Surely, this finishes the proof of \eqref{iii}, since
$$
|g_k^1(x)|\leq |f_k (x)|+b_{F_k} (x) |Dh(0)|+\omega_{F_k} ( |\overline{l}(x) | )\, d_{F_k}(x)+ \mu_{F_k}|Dh(0)|^2,
$$
then using that $|\overline{l}(x)|\leq |h(0)|+|Dh(0)|\,|x|\leq \|h\|_{C^{1,\bar{\alpha}} (\overline{B}_{1/2})}\leq K_2$ for all $ x \in B_1$,
we have
\begin{align*}
\|\,g_k\|_{L^p(B_1)}& \leq \|f_k\|_{L^p(B_1)} +\|b_{F_k}\|_{L^p(B_1)} K_2 + (K_2+1)\,\omega_{{F_k}}(1) \|d_{F_k}\|_{L^p(B_1)} \\
&+\mu K_2^2 \, |B_1 |^{\frac{1}{p}}+(1+ \|w\|_{L^\infty ({B_1})})\,\omega_{F_k} (1)\|d_{F_k}\|_{L^p(B_1)} \leq (5+2K_2)\,\delta \leq \gamma^{\alpha}
\end{align*}
from the definition of $\delta$ in \eqref{delta}. Thus, using the estimate above and \eqref{wlinfty local} in the $C^\beta$ local estimate, properly scaled to the ball of radius $\gamma$, we obtain in particular that
\begin{align*}
[w]_{\beta,\overline{B_\gamma}} &\leq \gamma^{-\beta} \widetilde{K}_1 \,\{ \, \|w\|_{L^\infty ({B_{2\gamma}})} +\gamma^{2-\frac{n}{p}} \|g_k\|_{L^p ({B_{2\gamma}})} \,\} \\
&\leq  \gamma^{-\beta}  {K}_1 \,\{ \,\gamma^{1+\alpha} + \gamma^{2-\frac{n}{p}}\gamma^\alpha\,\}\leq 2  K_1 \,\gamma^{1+\alpha-\beta}
\end{align*}
and so
$
\|w\|_{C^\beta (\overline{B_\gamma})}=\|w\|_{L^\infty ({B_\gamma})} +
[w]_{\beta,\overline{B_\gamma}}
\leq \gamma^{1+\alpha} +2K_1 \,\gamma^{1+\alpha-\beta}
\leq (1+2K_1)\,\gamma^{1+\alpha-\beta}
$
as desired.
\end{proof}

\begin{rmk}\label{obs plafond C1,alpha}
By the proof above we see that, under $\mu, \|b\|_{L^p(\Omega)}, \omega(1) \|d\|_{L^\infty(\Omega)} \leq C_1$, both $\sigma$ and the final constant $C$ depends on $n,p,\lambda,\Lambda, \alpha,\beta, K_1,K_2,C_0$ and $C_1$.
This is very useful in applications, when we have, for example, a sequence of solutions $u_k$ with their respective coefficients uniformly bounded; with $\|u_k\|_{L^\infty}$ and the $L^p$ norm of the right hand side a priori bounded. Then we can uniformly bound the $C^{1,\alpha}$ norm of $u_k$.
\end{rmk}

\subsection{Boundary Regularity}\label{boundary regularity}

Since our equation is invariant under diffeomorphisms and $\partial\Omega\in C^{1,1}$, we only need to prove regularity and estimates for some half ball, say $B_1^+(0)$.
Indeed, near a boundary point we make a diffeomorphic change of independent variable, which takes a neighborhood of $\partial\Omega$ into $B_1^+$. This change only depends on the coefficients of the equation and the $C^{1,1}$ diffeomorphisms that describe the boundary, see details in \cite{Winter} (see also \cite{tese} for a version with superlinear growth).

\vspace{0.1cm}

Then, consider $K_1$ and $\beta$ the pair of $C^\beta$ global superlinear estimate (proposition \ref{Cbetaquad}) in $B_1^+$, related to the initial $n,p,\lambda, \Lambda,\mu$, $\|b\|_{L^p(\Omega)}, \tau$ and $C_1$, such that $$\|u\|_{C^\beta(\overline{B}_1^+)}\leq K_1 \,\{\|u\|_{L^\infty (B_1^+)} + \|f\|_{L^p(B_1^+)} + \|u\|_{C^\tau (\mathbb{T})} +\|d\|_{L^p(B_1^+)}\,\omega(\|u\|_{L^\infty (B_1^+)})\}.$$

As in \cite{Winter}, we start proving a boundary version of the approximation lemma in $B_1^\nu$.
For this set, let $K_3\geq 1$ and $\bar{\alpha}$ be the pair of $C^{1,\bar{\alpha}}$ boundary estimate (proposition \ref{C1,baralphaglobal}) associated to $n,\lambda,\Lambda$ and $\tau$, independently of $\nu>0$.

We can suppose that $K_1\geq \widetilde{K}_1$ and $\beta\leq \widetilde{\beta}$, where $\widetilde{K}_1,\widetilde{\beta}$ is the pair of $C^\beta$ global estimate for the set $B_1^\nu$ (or $B_{1/2}^\nu$), independently of $\nu>0$, with respect to an equation with given constants $n,p,\lambda,\Lambda$ and bounds for the coefficients $\mu\leq 1$, $\|b\|_{L^p(B_2^\nu)}\leq 1+2K_3\, (3+2C_n) |B_1|^{1/p}$ (for a constant $C_n$, depending only on $n$, from lemma 6.35 of \cite{GT} for $\epsilon=1/2$, that will appear in the sequel) and $\omega(1)\|d\|_{L^p(B_2^\nu)}\leq 1$, for any solution in $B_2^\nu$ satisfying $\|u\|_{L^\infty(B_2^\nu)}\leq 1$ and $\|\psi\|_{C^{1,\tau} (\mathbb{T}_2^\nu)}\leq 2$ (or for any solution in $B_1^\nu$ with coefficients in $B_1^\nu$).
Denote $ \|\cdot\|_{L^p_\nu} =\|\cdot \|_{ L^p (B_1^\nu)}$.

\begin{lem}\label{AproxLemBoundary}
Assume $F$ satisfies \ref{SCmu} in $B_1^\nu$ for some $\nu \in [0,1]$ and $f \in L^p (B_1^\nu)$, where $p>n$. Let $\psi \in C^\tau (\partial B_1^\nu)$ with $\|\psi\|_{C^\tau (\partial B_1^\nu)}\leq K_0$. Set $L(x)=A+B\cdot x$ in $B_1^\nu$, for $A\in \real$, $B\in \rn$.
Then, for every $ \varepsilon>0$, there exists $\delta\in (0,1)$, $\delta=\delta (\varepsilon,n,p,\lambda,\Lambda,\tau,K_0)$, such that if
\begin{align*}
\|{\bar{\beta}}_F(\cdot,0)\|_{ L^p_\nu}\, ,\;\|f\|_{L^p_\nu}\, ,\; \mu(|B|^2 +|B|+1) \, ,\;
\|b\|_{L^p_\nu}(|B|+1) \, ,\;\omega (1) \|d\|_{L^p_\nu}(|A|+|B|+1)\leq \delta
\end{align*}
then any two $L^p$-viscosity solutions $v$ and $h$ of
\begin{align*}
\left\{
\begin{array}{rclcc}
F(x,v+L(x),Dv+B,D^2v)&=&f(x) & \mbox{in} & B_1^\nu \\
v &=& \psi & \mbox{on} & \partial B_1^\nu
\end{array}
\right.
,\;
\left\{
\begin{array}{rclcc}
F(0,0,0,D^2h)&=& 0 & \mbox{in} & B_1^\nu \\
h &=& \psi & \mbox{on} & \partial B_1^\nu
\end{array}
\right.
\end{align*}
respectively, with $\omega (1) \|d\|_{L^p_\nu}\|v\|_{L^\infty_\nu}\leq \delta$, satisfy $\|v-h\|_{L^\infty (B_1^\nu)}\leq \varepsilon$.
\end{lem}

\begin{proof}
For $\varepsilon>0$, we will prove the existence of $\delta\in (0,1)$ as above with $\delta\leq 2^{-\frac{n}{2p}}C_n^{-\frac{1}{2}}\,\widetilde{\delta}^{1/2}$, for $\widetilde{\delta}$ as in lemma \ref{AproxLem}.
Suppose the contrary, then there exist $\varepsilon_0>0$ and sequences $\nu_k \in [0,1]$, $F_k$ satisfying $(SC)^{\mu_k}$ for $b_k,\, d_k\in L^p_+(B_1^{\nu_k} )$, $\mu_k\geq 0$, $\omega_k$ modulus; $f_k\in L^p(B_1^{\nu_k})$, $A_k\in \real$, $B_k\in\rn$, $L_k(x)=A_k+B_k\cdot x$, and $\delta_k\in(0,1)$ with $\delta_k\leq 2^{-\frac{n}{2p}}C_n^{-\frac{1}{2}}\,\widetilde{\delta_k}^{1/2}$ for $\widetilde{b}_k=b_k+2|B_k|\mu_k$, such that
\begin{align*}
\|\bar{\beta}_{F_k}(\cdot,0)\|_{ L^p_{\nu_k} },\; \|f_k\|_{ L^p_{\nu_k} } ,\; \mu_k (|B_k|^2 +|B_k|+1),\;
\|b_k\|_{ L^p_{\nu_k} } (|B_k|+1) ,\; \omega_k(1) \|d_k\|_{ L^p_{\nu_k} } (|A_k|+|B_k|+1)
\end{align*}
are less or equal than $\delta_k $ with $ \delta_k \rightarrow 0$, and $v_k$, $h_k\in C(\overline{B_1^{\nu_k}})$ are $L^p$-viscosity solutions of
\begin{align*}
\left\{
\begin{array}{rclc}
F_k(x,v_k+L_k(x),Dv_k+B_k,D^2v_k)&=& f_k(x) & B_1^{\nu_k} \\
v_k &=& \psi_k & \partial B_1^{\nu_k}
\end{array}
\right.
,\;
\left\{
\begin{array}{rclc}
F_k(0,0,0,D^2h_k)&=&0  & B_1^{\nu_k} \\
h_k &=& \psi_k  & \partial B_1^{\nu_k}
\end{array}
\right.
\end{align*}
where $\|\psi_k\|_{C^\tau (\partial B_1^{\nu_k})}\leq K_0$, $\omega_k (1) \|d_k\|_{L^p_{\nu_k}}\|v_k\|_{L^\infty_{\nu_k}} \leq \delta_k$, but $
\|v_k-h_k\|_{L^\infty (B_1^{\nu_k})} > \varepsilon_0.$

\vspace{0.1cm}
Analogously to the proof of lemma \ref{AproxLem}, ABP implies that
$\|v_k\|_{L^\infty (B_1^{\nu_k})}\,, \;\|h_k\|_{L^\infty (B_1^{\nu_k})}\leq C_0 $
for large $k$, where $C_0$ is a constant that depends only on $n,p,\lambda,\Lambda$ and $K_0$.

Notice that $B_1^{\nu_k}$ has the exterior cone property, then by the $C^\beta$ global quadratic estimate (proposition \ref{Cbetaquad}) we obtain $\beta\in (0,1)$ such that
\begin{align} \label{CbetaApBoundary}
\|v_k\|_{C^\beta (\overline{B_1^{\nu_k}})}\,, \;\|h_k\|_{C^\beta (\overline{B_1^{\nu_k}})}\leq C \, , \quad \textrm{for large }k,
\end{align}
where $\beta=\min {(\beta_0,{\tau}/{2})}$ for some $\beta_0=\beta_0 (n,p, \lambda, \Lambda)$ and $C=C(n,p, \lambda, \Lambda, C_0)$. Observe that $\beta$ and $C$ do not depend on $k$, since $\mu_k,  \,\|\tilde{b}_k\|_{L^p(B_1^{\nu_k})},\, \omega_k(1) \,\|d_k\|_{L^p(B_1^{\nu_k})}, \,\|f_k\|_{L^p(B_1^{\nu_k})}\leq 1$ and $\mathrm{diam} (B_1^{\nu_k}) \leq  2$, for large $k $. Here we have different domains, what prevents us from directly using the compact inclusion $C^\beta$ into the set of continuous functions in order to produce convergent subsequences. But this is just a technicality, as in \cite{Winter}, by taking a subsequence of $\nu_k$ that converges to some $\nu_\infty \in [0,1]$, which we can suppose monotonous. Hence we consider two cases: $B_1^{\nu_\infty}\subset B_1^{\nu_k} \subset B_1^{\nu_{k+1}}\subset ...$ or $...\subset B_1^{\nu_{k+1}}\subset B_1^{\nu_k} \subset B_1^{\nu_\infty}$, for all $k \in \n$.
In the first one, we use the compact inclusion on $\overline{B_1^{\nu_\infty}}$. In the second, we make a trivial extension of our functions to the larger domain $\overline{B_1^{\nu_\infty}}$, i.e. by defining $\psi_k$ in $\widetilde{B}_k=B_1 \cap \{ - \nu_\infty \leq x_n \leq - \nu_k \}$ in such a way that $\|\psi_k\|_{C^\tau (\widetilde{B}_k)}\leq C_0$, from where we may suppose that \eqref{CbetaApBoundary} holds on $\overline{B_1^{\nu_\infty}}$ for the extended $v_k$ and $h_k$.
In both cases, we obtain convergent subsequences
$v_k\longrightarrow v_\infty$, $ h_k\longrightarrow h_\infty$ in $C(\overline{B_1^{\nu_\infty}})$ as ${k\rightarrow\infty}$,
for some continuous functions $v_\infty, \; h_\infty$ in $\overline{B_1^{\nu_\infty}}$, with $v_\infty = h_\infty = \psi_\infty$ on $\partial B_1^{\nu_\infty}$.

Finally, we claim that $v_\infty$ and $h_\infty$ are viscosity solutions of
\begin{align*}
\left\{
\begin{array}{rclcc}
F_\infty (D^2 u)&=&0  & \mbox{in} & B_1^{\nu_\infty} \\
u &=& \psi_\infty \; & \mbox{on} & \partial B_1^{\nu_\infty}
\end{array}
\right.
\end{align*}
and therefore equal by proposition \ref{ExisUnicF(D2u)}, which contradicts $\|v_\infty-h_\infty\|_{L^\infty (B_1^{\nu_\infty})}\geq \varepsilon_0$.

For $h_\infty$, it follows by taking the uniform limits on the inequalities satisfied by $h_k$. For $v_\infty$, we apply proposition \ref{Lpquad} together with observation \ref{Lpquadencaixados}, since we have, for each $\varphi\in C^2(D)$, where $D\subset B_1^{\nu_\infty}$, that
$F_k(x,v_k+L_k(x),D\varphi+B_k, D^2 \varphi)-f_k(x)-F_\infty (D^2 \varphi)\rightarrow 0$ as ${k\rightarrow\infty}$
in $L^p(D)$, analogously to the end of the proof of lemma \ref{AproxLem}.
\end{proof}

\begin{proof}[\textit{Proof of Boundary Regularity Estimates in the set $B_1^+$.}]
We proceed as in the local case, introducing the corresponding changes, in order to deal with the boundary. Our approach is similar to \cite{Winter}. Now we set
$W:=\| u \|_{L^{\infty} (B_1^+)} + \|f \|_{L^p (B_1^+ )} + \|u\|_{C^{1,\tau} (\mathbb{T} )} +\|d\|_{L^p(B_1^+)}\,\omega(\|u\|_{L^\infty (B_1^+)})  \leq W_0$ and $s_0:=\min (r_0, \frac{1}{2})$.

\vspace{0.1cm}

Fix $\alpha\in (0,\bar{\alpha})$ with $\alpha\leq\min (\beta,1- \frac{n}{p},\tau,\bar{\alpha} (1-\tau))$ and choose $\gamma=\gamma(n,\alpha,\bar{\alpha},K_3)\in (0,\frac{1}{4}]$ such that
$2^{2+\bar{\alpha}}K_4 \, \gamma^{\bar{\alpha}}\leq \gamma^{\alpha}$,
where $K_4=K_4 \,(K_3,n)\geq 1$ will be specified later. Thus, define
$\varepsilon=\varepsilon (\gamma)$ by $K_4 \, (2\gamma)^{1+\bar{\alpha}}.$
This $\varepsilon$ provides a $\delta=\delta(\varepsilon)\in (0,1)$, the constant of the approximation lemma~\ref{AproxLemBoundary} which, up to diminishing, can be supposed to satisfy
$(5+2K_4) \,\delta \leq \gamma^\alpha.$
Next we chose $\sigma=\sigma (s_0,n,p,\alpha , \bar{\alpha},\beta, \delta,\mu, \|b\|_{L^p(B_1^+)},\omega(1)\|d\|_{L^\infty (B_1^+)},K_1,K_3,C_0) \leq \frac{s_0}{2}$ such that
$$\sigma^{\min{({1-\frac{n}{p}},\beta)}} m \leq  {\delta} \,\{{32K^2 (K_4+K+1)|B_1|^{1/p}}\}^{-1}$$
where $m:=\max{ \{ 1, \|{b}\|_{L^p (B^+_1)}, \omega(1)\|{d}\|_{L^\infty (B^+_1)},\mu (1+2^\beta K_1) W_0 \} }$ and $K:={K_4}\,{{\gamma^{-\alpha}}(1-\gamma^\alpha)^{-1}} +{K_4}\,{{\gamma^{-1-\alpha}}(1-\gamma^{1+\alpha})^{-1}}\geq K_4\geq 1$.

\vspace{0.1cm}

Fix $z=(z^\prime,z_n) \in B_{1/2}^+ (0)$. We split our analysis in two cases, depending on the distance of the point $z$ to the bottom boundary: 1) $z_n < \frac{\sigma}{2}\; \Leftrightarrow\; \nu < \frac{1}{2}$ and 2) $z_n \geq  \frac{\sigma}{2} \;\Leftrightarrow\; \nu \geq \frac{1}{2}$, for $\nu :=\frac{z_n}{\sigma}$.

Suppose the first one. In this case we will be proceeding as in \cite{Winter} by translating the problem to the set $B_{2}^\nu $, in order to use the approximation lemma in its boundary version \ref{AproxLemBoundary}. Notice  that
$$
x \in B_2^\nu (0) \; \Leftrightarrow\; \sigma x +z \in {B_{2 \sigma}^+} (z) \subset B_1^+(0).
$$
\begin{figure}[!htb]
\centering
\includegraphics[scale=0.28]{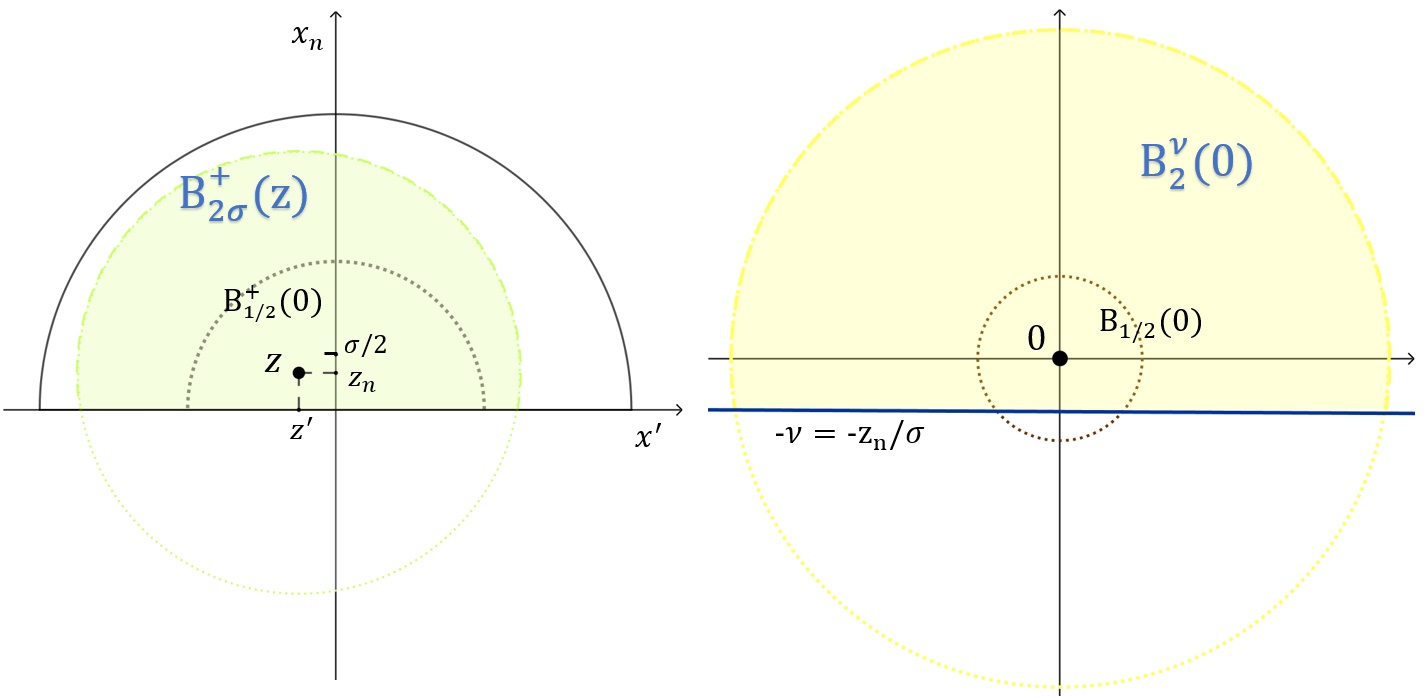}
\caption{ Illustration of the change of variable, from $B^+_{2\sigma}(z)=B_{2\sigma}(z)\cap\{x_n>0\}$, which is a subset of $B_1^+(0)$, to $B_2^\nu(0)=B_2(0)\cap\{x_n>-\nu\}$.}
\label{Rotulo}
\end{figure}

Then we define
$
N=N_\sigma (z):= \sigma W + \sup_{x\in B_2^\nu (0)} |u(\sigma x +z) - u(z)|.
$
The $C^\beta$ quadratic estimate, this time the global one, restricted to the set $B_{2 \sigma}^+ (z)$, yields
\begin{align}\label{Nboundary}
\sigma W \leq N \leq (\sigma + 2^\beta K_1 \sigma ^\beta)W
\leq (1+2^\beta K_1) \sigma ^\beta W_0.
\end{align}
Next we set $\widetilde{u}(x)=\frac{1}{N}\{u(\sigma x +z)-u(z)\}$, which is, as in claim \ref{claim C1,alpha local 1a mud.var.}, an $L^p$-viscosity solution of
\begin{align*}
\left\{
\begin{array}{rclcc}
\widetilde{F}(x,\widetilde{u},D \widetilde{u},D^2 \widetilde{u}) &=& \widetilde{f} (x) &\mbox{in} & B_2^\nu  \\
\widetilde{u} &=& \widetilde{\psi} &\textrm{on} & \mathbb{T}^{\nu}_2
\end{array}
\right.
\end{align*}
for
$$\widetilde{F}(x,r,p,X):= \frac{\sigma^2}{N} F\left( \sigma x +z,Nr+u(z),\frac{N}{\sigma}p, \frac{N}{\sigma^2} X \right)-\frac{\sigma^2}{N}
F(\sigma x+z, u(z),0,0),$$
$\widetilde{\psi}(x):=\frac{1}{N} \{ \psi (\sigma x +z)-u(z) \}$ and $\widetilde{f}:=\widetilde{f}_1+\widetilde{f}_2$ where
\begin{align*}
\widetilde{f}_1(x):={\sigma^2} f(\sigma x+z)/{N};\;\;\widetilde{f}_2(x):= -{\sigma^2}F(\sigma x+z, u(z),0,0)/{N},
\end{align*}
$\widetilde{F}$ satisfying $(\widetilde{SC})^{\widetilde{\mu}}$ for $\widetilde{b}(x)= \sigma b(\sigma x+z)$, $\widetilde{\mu}=N\mu$, $\widetilde{d}(x)= \sigma^2 d(\sigma x+z)$ and $\widetilde{\omega}(r)=\omega(Nr)/N$.
\vspace{0.1cm}

With this definition and the choice of $\sigma$ in \eqref{sigma}, we obtain $\|\widetilde{u}\|_{L^\infty (B_2^\nu )} \leq 1$, $\|\widetilde{f}\|_{L^p (B_2^\nu)}\leq \frac{\delta}{8}$, $\widetilde{\mu} \leq \frac{\delta}{8K^2 |B_1|^{1/p} }$,
$\|\widetilde{b}\|_{L^p (B_2^\nu)}\leq \frac{\delta}{16K} $,
$\widetilde{\omega} (1) \|\widetilde{d}\|_{L^p (B_2^\nu)}\leq \frac{\delta}{32(K_4+K+1)}$ and $\|\bar{\beta}_{\widetilde{F}}(0,\cdot)\|_{L^p (B_1^\nu)}\leq \delta/4 $ by choosing $\theta=\delta/8$, as in the local case.

\vspace{0.1cm}
Furthermore, we have $\|\widetilde{\psi} \|_{L^\infty (\mathbb{T}^{\nu}_2)}  \leq\|\widetilde{u}\|_{L^\infty (B_2^\nu )} \leq 1$
and then $\|D \widetilde{\psi}\|_{C^\tau (\mathbb{T}^{\nu}_2)}$ is bounded by
\begin{align*}
 \frac{\sigma}{N}  \| D\psi \|_{L^\infty (B_{2\sigma} (z) \cap \mathbb{T})} +  \frac{\sigma}{N}  \sup_{\substack{ x\neq y \in \mathbb{T}_2^\nu }} \frac{|D\psi (\sigma x +z)-D\psi (\sigma y+z)|}{|\sigma x-\sigma y|^\tau} \sigma^\tau\leq \frac{ \|\psi \|_{C^{1,\tau} (\mathbb{T})}}{W}\leq 1
\end{align*}
since $N\geq \sigma W$. Therefore, we obtain $\|\widetilde{\psi} \|_{C^{1,\tau} (\mathbb{T}^{\nu}_2)}=\|\widetilde{\psi} \|_{L^\infty (\mathbb{T}^{\nu}_2)} +\|D \widetilde{\psi}(x)\|_{C^\tau (\mathbb{T}^{\nu}_2)} \leq 2$.\medskip

We can suppose, up to this rescaling, that $F,u, \mu,b,d,\omega$ satisfy the former hypotheses related to $\widetilde{F},\widetilde{u},\widetilde{\mu},\widetilde{b},\widetilde{d},\widetilde{\omega}$.  Thus, we move to the construction of $l_k (x):=a_k +b_k \cdot x$ such that
\vspace{0.2cm}

$(i)_k \;\;\| u-l_k \|_{L^\infty (B_{r_k}^\nu)} \leq r_k^{1+\alpha}$
\vspace{0.2cm}

$(ii)_k \;\;|a_k-a_{k-1}| \leq K_4 \,r_{k-1}^{1+\alpha}\;,\;\; |b_k - b_{k-1}|\leq K_4\, r_{k-1}^\alpha$
\vspace{0.2cm}

$(iii)_k \;\;|(u-l_k)(r_k x)-(u-l_k)(r_k y)|\leq C_{1,4} \, r_k^{1+\alpha} |x-y|^\beta \;\; \textrm{ for all } x,y\in B_1^{\nu_k}$
\vspace{0.2cm}\\
where $C_{1,4}=C_{1,4}\,(K_1,K_4)$ and $\nu_k := \frac{\nu}{r_k}$, $r_k=\gamma^k$ for some $\gamma\in(0,1)$, for all $k\geq 0$ ($l_{-1}\equiv 0$). \medskip

We emphasize that these iterations will prove that the function $u$ (which plays the role of $\widetilde{u}$) is differentiable at $0$ and provide
$| u(x) - u(0) - D u(0)\cdot x | \leq C|x|^{1+\alpha} $, $|D u(0)| \leq C$
for every $x\in B_1^\nu$.
In terms of our original function defined on $B_1^+$, it means that $u$ will be differentiable at $z$, for all $z$ with $z_n<\frac{\sigma}{2}$. On the other hand, the second case $z_n\geq \frac{\sigma}{2}$ is covered by the local part, section \ref{local regularity}, since in this situation we are far away from the bottom boundary. Consequently, boundary superlinear regularity and estimates on $B_1^+$ will follow by a covering argument.

For the proof of $(i)_k-(iii)_k$, we use induction on $k$. For $k=0$ we set $a_0 = b_0 = 0$. Recall that $\beta$ and $K_1$ are the constants from $C^\beta$ quadratic global estimate in the set $B_1^{\nu}$, then we have
$
\|u\|_{C^\beta(\overline{B_1^\nu})}\leq \widetilde{K}_1 (1+\delta+2+1)\leq 5K_1
$
and so $(iii)_0$ for $\nu_0=\nu$, $(i)_0$ and $(ii)_0$ are valid.

Analogously to the the local case, we have
$|b_k|, \; \|l_k\|_{L^\infty (B_{r_k}^\nu)} \leq  K .$
For the induction's step we suppose $(i)_k-(iii)_k$ and construct $a_{k+1}$, $b_{k+1}$ such that $(i)_{k+1}-(iii)_{k+1}$ are valid.
Define
$$
v(x)=v_k({x}) := \frac{(u-l_k)(r_k {x})}{r_k^{1+\alpha}}=\frac{u(r_k{x}) - a_k - b_k\cdot x r_k}{r_k^{1+\alpha}}\, , \;\textrm{ for all }\,{x}\in B_2^{\nu_k} .
$$
Since $r_k x\in B_{r_k}^{\nu} \Leftrightarrow x \in B_1^{\nu_k}$, $(i)_k$ says that $ |v| \leq 1$ in $B_1^{\nu_k}.$ From this and  $(iii)_k$, we get
$$
\|v\|_{C^\beta (\overline{B_1^{\nu_k}})}=\|v\|_{L^\infty ({B}_1^{\nu_k})}+ \sup_{\substack{ x,y \in B_1^{\nu_k} \\ x \neq y }} \frac{|v(x)-v(y)|}{|x-y|^\beta}  \leq 1+C_{1,4}=:K_0.
$$
Notice that, as in the local case, it follows from claim \ref{claim C1,alpha local 2a mud.var.} that $v$ is as an $L^p$-viscosity solution of
\begin{align*}
\left\{
\begin{array}{rclcc}
F_k(x,v+L_k (x), Dv+B_k, D^2 v) &=& f_k (x) &\mathrm{in} & B_2^{\nu_k} \\
v & =& \psi_k &\mathrm{on} & \mathbb{T}_2^{\nu_k}
\end{array}
\right.
\end{align*}
where 
$f_k(x):=r_k^{1-\alpha} f(r_k x)$, $L_k(x)=A_k+B_k\cdot x$ in $B_2^{\nu_k}$ for $A_k=r_k^{-1-\alpha} a_k$, $B_k=r_k^{-\alpha} b_k$, and
\begin{align*}
F_k(x,s,p,X):=r_k^{1-\alpha} F(r_k x,r_k^{1+\alpha} s, r_k^\alpha p, r_k^{\alpha -1}X)
\end{align*}
satisfying $(SC)_{F_k}^{\mu_{F_k}}$ for $b_{F_k}(x)=r_k b(r_k x)$, $\mu_{F_k}=r_k^{1+\alpha}\mu$, $d_{F_k}(x)=r_k^2 d(r_k x)$, and $\omega_{F_k}(s)=r_k^{-1-\alpha}\omega(r_k^{1+\alpha} s)$, for the Lipschitz modulus $\omega$. 

\vspace{0.1cm}

The above coefficients satisfy the hypotheses of lemma \ref{AproxLemBoundary}, since $\|b_{F_k}\|_{L^p (B_1^{\nu_k})}(|B_k|+1) \leq  \delta$, $\mu_{F_k}(|B_k|^2 +|B_k|+1)\leq  \delta$, $\omega_{F_k}(1)\|d_{F_k}\|_{L^p(B_1^{\nu_k})} (|A_k|+|B_k|+1)\leq  \delta$,  $\|v\|_{\infty} \leq 1$, $\|f_k\|_{L^p(B_1^{\nu_k})}\leq \delta$,
and $\|\bar{\beta}_{F_k}(\cdot,0)\|_{L^p(B_1^{\nu_k})}\leq \delta$, see section \ref{local regularity}. 

\vspace{0.1cm}

Let $h=h_k\in C(\overline{B_1^{\nu_k}})$ be the $C$-viscosity solution of
\begin{align*}
\left\{
\begin{array}{rclcc}
F_k (0,0,0,D^2 h)&=& 0 &\mbox{in} & B_1^{\nu_k}   \\
h &=& v \quad &\mbox{on} & \partial B_1^{\nu_k}
\end{array}
\right.
\end{align*}
given by proposition \ref{ExisUnicF(D2u)}, since $B_1^{\nu_k}$ has the uniform exterior cone condition. From ABP we get $\|h\|_{L^\infty (B_1^{\nu_k})}\leq \|h\|_{L^\infty (\partial B_1^{\nu_k})}\leq 1$. Further, $h=v=\psi_k \in C^{1,\tau}(B_1\cap \{x_n=-\nu_k \})$ and we can find a uniform bound for the $C^{1,\tau}$ norm of $\psi_k$. Indeed, $\|\psi_k\|_{L^\infty (\mathbb{T}_1^{\nu_k} )}\leq \|v\|_{L^\infty (\overline{B_1^{\nu_k}} )} \leq 1$ and
$$
[D\psi_k]_{\tau,\mathbb{T}_1^{\nu_k} } =\sup_{\substack{ x,y \in \mathbb{T}_1^{\nu_k} \\ x \neq y }} \frac{|D\psi_k (x)-D\psi_k (y)|}{|x-y|^\tau} = \sup_{\substack{ \tilde{x}, \tilde{y} \in \mathbb{T}_{r_k}^{\nu} \\  \tilde{x}=r_k x, \,\tilde{y}= r_k y }} \frac{|D\psi (\tilde{x})-D\psi (\tilde{y})|}{|\tilde{x}-\tilde{y}|^\tau} r_k^{\tau - \alpha} \leq  1
$$
since $\|D\psi\|_{C^{\tau} (\mathbb{T}_1^{\nu} )} \leq  1$ and $\alpha \leq \tau$. Moreover, using the global Holder interpolation in smooth domains, lemma 6.35 of \cite{GT}, for $\epsilon = \frac{1}{2}$, there exists a constant\footnote{The proof of lemma 6.35 in \cite{GT} is based on an interpolation inequality (6.89) for adimensional Holder norms (that does not depend on the domain); followed by a partition of unity that straightens the boundary (not necessary in our case $\mathbb{T}_1^{\nu_k}\subset \real^{n-1} $). Then we have an estimate independently on $k$.}  $C_n$ such that
$$
\|\psi_k\|_{C^{1} (\mathbb{T}_1^{\nu_k} )} \leq C_n \, \|\psi_k\|_{C (\mathbb{T}_1^{\nu_k} )} +\frac{1}{2} \|\psi_k\|_{C^{1,\tau} (\mathbb{T}_1^{\nu_k} )}
$$
hence
$$
\|\psi_k\|_{C^{1,\tau} (\mathbb{T}_1^{\nu_k} )} =\|\psi_k\|_{C^{1} (\mathbb{T}_1^{\nu_k} )} + [D\psi_k]_{\tau,\mathbb{T}_1^{\nu_k}} \leq C_n +\frac{1}{2} \|\psi_k\|_{C^{1,\tau} (\mathbb{T}_1^{\nu_k} )} +1
$$
i.e. $ \|\psi_k\|_{C^{1,\tau} (\mathbb{T}_1^{\nu_k} )} \leq 2(C_n+1)$.
Thus, the $C^{1,\bar{\alpha}}$ global estimate (proposition \ref{C1,baralphaglobal}) yields
$$
\|h\|_{C^{1,\bar{\alpha}} (\overline{B_{1/2}^{\nu_k}})}\leq K_3 \,\{ \|h\|_{L^\infty (B_1^{\nu_k})} + \|\psi_k\|_{C^{1,\tau} (\mathbb{T}_1^{\nu_k} )}  \} \leq K_3\, (3+2C_n)=: K_4.
$$
Now, the approximation boundary lemma \ref{AproxLem} applied to $F_k,v,h,\nu_k, \mu_{F_k},b_{F_k}, d_{F_k},\omega_{F_k}, A_k, B_k,  \psi_k $,  $\beta, K_0$ gives us that
$
\|v-h\|_{L^\infty (B_1^{\nu_k})}\leq \varepsilon.
$

Therefore, defining $\overline{l}(x)=\overline{l}_k(x):=h(0)+Dh(0)\cdot x$ in $B_{1}^{\nu_k}$, we have
\begin{align} \label{wlinftyboundary}
\|v-\overline{l}\|_{L^\infty (B_{2\gamma}^{\nu_k})}\leq \gamma^{1+\alpha}.
\end{align}
In fact, by the choice of $\gamma$ we have, for all $x\in B_{2\gamma}^{\nu_k}(0)$,
\begin{align*}
|v(x)-\overline{l}(x)|\leq  |v(x)-h(x)| + |h(x) - h(0) - D h(0)\cdot x|\leq  2 K_4 \,(2\gamma)^{1+\bar{\alpha}}\leq \gamma^{1+\alpha}.
\end{align*}

Next, \eqref{wlinftyboundary} and the definition of $v$ imply
\begin{align*}
|u(r_k {x}) -l_k(r_kx) - r_k^{1+\alpha} h(0) - r_k^{1+\alpha} D h(0) \cdot  x |\leq r_k^{1+\alpha} \gamma^{1+\alpha} = r_{k+1}^{1+\alpha}\, \; \textrm{ for all } x \in B_{2\gamma}^{\nu_k} \, ,
\end{align*}
which is equivalent to
\begin{align*}
|u(y) -l_{k+1}(y)|\leq r_k^{1+\alpha} \gamma^{1+\alpha} = r_{k+1}^{1+\alpha}\, \; \textrm{ for all } y=r_kx \in B_{2\gamma r_k}^{\nu}= B_{2r_{k+1}}^\nu\, ,
\end{align*}
where $l_{k+1}(y):=l_k(y)+r_k^{1+\alpha} h(0) + r_k^{\alpha} D h(0) \cdot  y $. Then, we define
$a_{k+1} := a_k + h(0) r_k^{1+\alpha}$, $b_{k+1} := b_k + D h(0) r_k^{\alpha}$,
obtaining $(i)_{k+1}$. Also, $|a_{k+1} - a_k|\leq K_4 \,r_k^{1+\alpha}$, $|b_{k+1} - b_k|\leq K_4\, r_k^{\alpha}$, which is $(ii)_{k+1}$. As in the local case, to finish the proof of $(iii)_{k+1}$, it is enough to show that
$$\|v-\overline{l}\|_{C^\beta (\overline{B_\gamma^{\nu_k}})}\leq C_{1,4} \,\gamma^{1+\alpha-\beta}.$$

Let us see that this is obtained by applying the global superlinear $C^\beta$ estimate in proposition \ref{Cbetaquad} to the function $w:=v-\overline{l}$.

Analogously to the local case, $w$ is an $L^p$-viscosity solution in $B_2^{\nu_k}$ of \eqref{Cbetafinal} (see notations and coefficients there), in addition to $w=\psi_k - \overline{l}$ on $ \mathbb{T}_2^{\nu_k}$.
The definition of $\delta$ gives us  $\|g_k\|_{L^p(B_1^{\nu_k})}\leq (5+2K_4)\delta \leq \gamma^{\alpha}$.
Further, using that $\psi_k=h$ on $\mathbb{T}_{2\gamma}^{\nu_k}$, we obtain $\|\psi_k-\overline{l}\|_{L^\infty (\mathbb{T}_{2\gamma}^{\nu_k})}\leq \gamma^{1+{\alpha}}.$

Now, since $\psi_k-\overline{l}\in C^1(\mathbb{T}_{2\gamma}^{\nu_k})$, it is a Lipschitz function with constant less or equal than $\|D\psi_k-D\overline{l}\,\|_{C (\mathbb{T}_{2\gamma}^{\nu_k})}\leq 2(C_n+1)+K_4\leq 2 K_4$ and thus
\begin{align*}
|(\psi_k-\overline{l})(x)-(\psi_k-\overline{l})(y)|&=|(\psi_k-\overline{l})(x)-(\psi_k-\overline{l})(y)|^\tau |(\psi_k-\overline{l})(x)-(\psi_k-\overline{l})(y)|^{1-\tau}\\
&\leq (2K_4)^\tau (2K_4)^{1-\tau} |x-y|^\tau \,\gamma^{(1+\bar{\alpha})(1-\tau)}= 2K_4\, |x-y|^\tau \,\gamma^{1-\tau+\bar{\alpha}(1-\tau)}.
\end{align*}
Then, the choice of $\alpha$ implies that
$
[\psi_k-\overline{l}]_{\tau,\mathbb{T}_{2\gamma}^{\nu_k}}\leq 4K_4\,\gamma^{1-\tau+\alpha}.
$
Hence, from this, \eqref{wlinftyboundary} and $C^\beta$ global estimate, properly scaled for the radius $\gamma$, we obtain
\begin{align*}
[w]_{\beta,\overline{B_\gamma^{\nu_k}}} &\leq \gamma^{-\beta} \widetilde{K}_1 \,\{ \, \|w\|_{L^\infty ({B_{2\gamma}^{\nu_k}})} +\gamma^{2-\frac{n}{p}} \|g_k\|_{L^p ({B_{2\gamma}^{\nu_k}})} +\|\psi_k-\overline{l}\|_{L^\infty (\mathbb{T}_{2\gamma}^{\nu_k}) }+\gamma^\tau [\psi_k-\overline{l}]_{\tau,\mathbb{T}_{2\gamma}^{\nu_k}} \,\} \\
&\leq  \gamma^{-\beta}  K_1 \,\{ \,2\gamma^{1+\alpha} + \gamma^{2-\frac{n}{p}}\gamma^\alpha+ 4K_4\,\gamma^{1+\alpha} \,\}\leq  K_1 \,(3+4K_4)\,\gamma^{1+\alpha-\beta}
\end{align*}
and finally, for $C_{1,4}:=1+(3+4K_4)K_1=C_{1,4}\,(K_1,K_4)$, we conclude
\begin{align*}
\|w\|_{C^\beta (\overline{B_\gamma^{\nu_k}})}=\|w\|_{L^\infty ({B_\gamma^{\nu_k}})} +
[w]_{\beta,\overline{B_\gamma^{\nu_k}}}
\leq \gamma^{1+\alpha} +(3+4K_4)K_1 \,\gamma^{1+\alpha-\beta}
\leq C_{1,4}\,\gamma^{1+\alpha-\beta} .
\end{align*}
\end{proof}

Therefore, the complete proof of regularity and estimates in the global case is done by a covering argument over the domain $\Omega$, using local and boundary results.

\section{$W^{2,p}$ Results}\label{W2,p regularity}

The first application of the $C^{1,\alpha}$ theory is $W^{2,p}$ regularity for solutions of fully nonlinear equations with superlinear growth in the gradient, which are convex or concave in the variable $X$. This extends the results in \cite{Winter} to superlinear growth in the gradient in the case $p>n$.

In the next two sections we make the convention that $\omega$ is a Lipschitz modulus in the sense that $\omega (r)\leq \omega (1) r$, for all $r\geq 0$, unless otherwise specified.

\begin{teo} \label{W2,p quad}{$(W^{2,p}$ Regularity\,$)$}
Let $\Omega\subset\rn$ be a bounded domain and $u\in C(\Omega)$ an $L^p$-viscosity solution of
\begin{align}\label{eqW2,p}
F(x,u,Du,D^2 u)+g(x,Du) =f(x) \quad \textrm{in}\;\;\;\Omega
\end{align}
where $f\in L^p(\Omega)$, $p>n$, $g$ is a measurable function in $x$ such that $g(x,0)=0$ and $|g(x,p)-g(x,q)|\leq \gamma|p-q|+ \mu |p-q|(|p|+|q|)$,
$F$ is convex or concave in $X$ satisfying $(SC)^{0}$, for $b,\, d \in L^\infty_+(\Omega)$ and $\omega$ a Lipschitz modulus. Also, suppose $\| u \|_{L^{\infty} (\Omega)} + \|f \|_{L^p (\Omega)} \leq C_0$.
Then, there exists $\theta=\theta (n,p,\lambda,\Lambda,\|b\|_{L^ p(\Omega)})$ such that, if \eqref{Htheta} holds for all $r\leq \min \{ r_0, \mathrm{dist} (x_0,\partial\Omega)\}$, for some $r_0>0$ and for all $x_0 \in \Omega$, this implies that $u\in W^{2,p}_{\mathrm{loc}} (\Omega)$ and for every $\Omega^\prime \subset\subset \Omega$,
\begin{align*}
\|u\|_{W^{2,p}(\Omega^\prime)} \leq C \,\{ \| u \|_{L^{\infty} (\Omega)} + \|f \|_{L^p (\Omega)} \}
\end{align*}
where $C$ depends on $\,r_0,n,p,\lambda,\Lambda, \mu,\| b \|_{L^p (\Omega)},\omega (1)\|d\|_{L^\infty(\Omega)},  \mathrm{dist} (\Omega^\prime,\partial\Omega),\mathrm{diam} (\Omega)$ and $C_0$.

If, moreover, $\partial\Omega\in C^{1,1}$, $u\in C(\overline{\Omega})$ and $u=\psi$ on $\partial\Omega$ for some $\psi \in W^{2,p} (\Omega )$ with $\| u \|_{L^{\infty} (\Omega)} + \|f \|_{L^p (\Omega)} + \|\psi\|_{W^{2,p}(\Omega)} \leq C_1$ then, there exists $\theta =\theta (n,p,\lambda , \Lambda,\|b\|_{L^ p(\Omega)})$ such that, if \eqref{Htheta} holds for some $r_0>0$ and for all $x_0 \in \overline{\Omega}$, this implies that $u\in W^{2,p}(\Omega)$ and satisfies the estimate
\begin{align*}
\|u\|_{W^{2,p}(\Omega)} \leq C\, \{ \| u \|_{L^{\infty} (\Omega)} + \|f \|_{L^p (\Omega)}  + \|\psi\|_{W^{2,p} (\Omega)} \}
\end{align*}
where $C$ depends on $r_0,n,p,\lambda,\Lambda, \mu,\| b \|_{L^p (\Omega)},\omega (1)\|d\|_{L^\infty(\Omega)},\partial\Omega,\mathrm{diam} (\Omega)$ and $C_1$.

\end{teo}

\begin{proof}
We prove only the global case, since in the local one we just ignore the term with $\psi$, by considering it equal to zero in what follows.
Notice that $\psi\in W^{2,p}(\Omega)\subset C^{1,\tau}(\overline{\Omega})$ for some $\tau\in (0,1)$ with continuous inclusion, then $\| u \|_{L^{\infty} (\Omega)} + \|f \|_{L^p (\Omega)} + \|\psi\|_{C^{1,\tau} (\partial\Omega)}\leq C_2$.

Thus, by $C^{1,\alpha}$ regularity theorem, we have that $\bar{f}(x):= f(x)-g(x,Du)\in L^p(\Omega)$ and also $$\|u\|_{C^{1,\alpha}(\overline{\Omega})}\leq C_3\, \{\| u \|_{L^{\infty} (\Omega)} + \|\bar{f} \|_{L^p (\Omega)} + \|\psi\|_{C^{1,\tau} (\partial\Omega)} \} .$$

\begin{claim}\label{claimW2,p}
$u$ is an $L^p$-viscosity solution of $F(x,u,Du,D^2u)=\bar{f}(x)$ in $\Omega$.
\end{claim}

\begin{proof}
Let us prove the subsolution case; for the supersolution it is analogous.
Assuming the contrary, there exists some $\phi\in W^{2,p}_{\mathrm{loc}}(\Omega)$, $x_0\in \Omega$ and $\varepsilon>0$ such that $u-\phi$ has a local maximum at $x_0$ and
$F(x,u,D \phi,D^2 \phi)-\bar{f}(x)\leq -\varepsilon$\; a.e. in $B_r(x_0)$.

In turn, by the definition of $u$ being an $L^p$-viscosity subsolution of \eqref{eqW2,p}, we have that
\begin{center}
$F(x,u,D \phi,D^2 \phi) +g(x,D\phi) \geq f(x)-\varepsilon /2$\; a.e. in $B_r(x_0)$
\end{center}
up to diminishing $r>0$. By subtracting the last two inequalities, we obtain that
\begin{align}\label{estW2,p absurdo}
-\{\gamma +\mu (|Du|+|D\phi|)\}\,|Du-D\phi|\leq g(x,Du)-g(x,D\varphi)\leq - \varepsilon /2 \, \textrm{ a.e. in } B_r(x_0).
\end{align}

Since  $u-\phi \in C^1(B_r(x_0))$ has a local maximum at $x_0$, we have $D (u-\phi)(x_0)=0$ and, moreover, $|D(u-\phi)(x)|<\varepsilon \,  \{\gamma+\mu (\|Du\|_{L^\infty (B_r(x_0))}+\|D\phi\|_{L^\infty (B_r(x_0)}) +1\}^{-1}/4\,$ for all $x\in B_r(x_0)$, possibly for a smaller $r$, which contradicts \eqref{estW2,p absurdo}.
\qedhere{\textit{Claim \ref{claimW2,p}.} }
\end{proof}

Thus by Winter's result, theorem 4.3 in \cite{Winter} (or Świech \cite{Swiech} in the local case), we have that $u\in W^{2,p}(\Omega)$ (respectively $u\in W^{2,p}_{\mathrm{loc}}(\Omega)$) and
\begin{align*}
&\|u\|_{W^{2,p}(\Omega)} \leq C\, \{ \| u \|_{L^{\infty} (\Omega)} + \|\bar{f} \|_{L^p (\Omega)}  + \|\psi\|_{W^{2,p} (\Omega)} \}\\
&\leq C\,\{ \| u \|_{L^{\infty} (\Omega)} + \|f \|_{L^p (\Omega)} +\mu \|u\|_{C^1(\overline{\Omega})}^2+\gamma \|u\|_{C^1(\overline{\Omega})} + \|\psi\|_{W^{2,p} (\Omega)} \}\\
&\leq C\,\{ \| u \|_{L^\infty(\Omega)} + \|f \|_{L^p(\Omega)} +  \|\psi\|_{W^{2,p}(\Omega)}  + (\mu\, C_2+\gamma) C_3\{\| u \|_{L^\infty (\Omega)}   +\|f \|_{L^p(\Omega)} + \|\psi\|_{C^{1,\tau} (\partial\Omega)} \}  \}
\end{align*}
which implies the estimate.
\end{proof}

In theorem \ref{W2,p quad}, the final constants only depend on the $L^p$-norm of the coefficient $b$, despite the boundedness hypothesis on it. The latter hypothesis is needed to conclude that solutions are twice differentiable a.e.
Observe that, in \cite{Winter} (see theorem 4.3 there), $W^{2,p}$ results consist of two parts: (i) introducing a new equation $F(x,0,0,D^2 u)=\tilde{f}(x)$ (via corollary 1.6 in \cite{Swiech}), in which $u$ remains a solution in the $L^p$-viscosity sense; (ii) obtaining $W^{2,p}$ estimates for solutions of $F(x,0,0,D^2u)=\widetilde{f}(x)$, which are independent of the zero and first order coefficients.

From the regularity and estimates related to $\mu =0$, we can give an alternative proof of proposition 2.4 in \cite{KSexist2009}, concerning existence and uniqueness for the Pucci's extremal operators with unbounded coefficients in the case $p>n$.

\begin{prop} \label{W2,p solv}{$($Solvability of the Dirichlet problem\,$)$}
Let $\Omega\subset\rn$ be a bounded $C^{1,1}$ domain. Let $b,\, d\in L^p_+ (\Omega)$, $p>n$ and $\omega$ a Lipschitz modulus.
Let $f\in L^p(\Omega)$ and $\psi\in W^{2,p}(\Omega)$.
Then, there exists $u_\pm \in C(\overline{\Omega})$ which are the unique $L^p$-viscosity solutions of the problems
\begin{align*}
\left\{
\begin{array}{rllll}
\mathcal{M}^\pm_{\lambda,\Lambda} (D^2 u_\pm)\pm b(x)|Du_\pm|\pm d(x)w((\mp\, u_\pm)^+) &=& f(x)&\mbox{in}&\Omega\\
u_\pm &=& \psi&\mbox{on}&\partial\Omega\, .
\end{array}
\right.
\end{align*}
Moreover, $u_\pm \in W^{2,p}(\Omega)$ and satisfies the estimate
\begin{align*}
\|u_\pm\|_{W^{2,p}(\Omega)} \leq C\, \{ \| u_\pm \|_{L^{\infty} (\Omega)} + \|f \|_{L^p (\Omega)}  + \|\psi\|_{W^{2,p} (\Omega)} \}
\end{align*}
where $C$ depends only on $n,p,\lambda,\Lambda,\| b \|_{L^p (\Omega)},\omega (1)\,\| d \|_{L^p (\Omega)},\partial\Omega$ and $\mathrm{diam} (\Omega)$.
\end{prop}

\begin{proof}
It is enough to treat the upper extremal case.
Let $b_k, \, d_k\in L^\infty_+ (\Omega)$ be such that $b_k\rightarrow b$ and $d_k\rightarrow d$ in $L^p(\Omega)$.
Let $u_k \in W^{2,p}(\Omega)$ be the unique $L^p$-viscosity solution of
\begin{align*}
\left\{
\begin{array}{rllll}
\mathcal{M}^+_{\lambda,\Lambda} (D^2 u_k)+ b_k(x)|Du_k|+ d_k(x)\,\omega(u_k^-) &=& f(x)&\mbox{in}&\Omega\\
u_k &=& \psi&\mbox{on}&\partial\Omega
\end{array}
\right.
\end{align*}
given by theorem 4.6 of \cite{Winter}.
From the estimates in theorem \ref{W2,p quad}, with $d_k(x)\omega (u_k^-)$ as RHS,
\begin{align} \label{ukW2,p}
\|u_k\|_{W^{2,p}(\Omega)} \leq C_k \, \{ \| u_k \|_{L^{\infty} (\Omega)} + \|f \|_{L^p (\Omega)}  + \|\psi\|_{W^{2,p} (\Omega)} \},
\end{align}
where $C_k$ remains bounded, since $b_k$ and $d_k$ are bounded in $L^p(\Omega)$.

Now, by ABP we have that that
$\|u_k\|_{L^\infty (\Omega)}\leq \|\psi\|_{L^\infty (\partial\Omega)} + C\, \|f\|_{L^p(\Omega)}$.
From this and \eqref{ukW2,p} we get
$\|u_k\|_{W^{2,p}(\Omega)} \leq C$ and hence there exists $u\in C^{1}(\overline{\Omega})$ such that $u_k\rightarrow u$ in $C^{1}(\overline{\Omega})$.

Next, proposition \ref{Lpquad} implies that $u$ is an $L^p$-viscosity solution of
\begin{align}\label{uLpvis}
\left\{
\begin{array}{rllll}
\mathcal{M}^+_{\lambda,\Lambda} (D^2 u)+ b(x)|Du|+ d(x)\,\omega(u^-) &=& f(x)&\mbox{in}&\Omega\\
u&=& \psi&\mbox{on}&\partial\Omega\, .
\end{array}
\right.
\end{align}
Notice that $W^{2,p}(\Omega)$ is reflexive, so there exists $\tilde{u}\in W^{2,p}(\Omega)$ such that $u_k$ converges weakly to $\tilde{u}$. By uniqueness of the limit, $\tilde{u}=u$ a.e. in $\Omega$, and $u$ is a strong solution of \eqref{uLpvis}.

Finally, if there would exist another $L^p$-viscosity solution of \eqref{uLpvis}, say $v\in C(\overline{\Omega})$, then the function $w:=u-v$ satisfies $w=0$ on $\partial\Omega$ and it is an $L^p$-viscosity solution of
$\mathcal{L}^+[w]\geq 0$ in $\Omega\cap\{w>0\}$.
Indeed, since $u$ is strong, we can apply the definition of $v$ as an $L^p$-viscosity supersolution with $u$ as a test function; we also use that $u^- \leq v^- + (u-v)^-$, monotonicity and subadditivity of the modulus.
Then, by ABP we have that $w\leq 0$ in $\Omega$. Analogously, from the definition of subsolution of $v$, we obtain $w\geq 0$ in $\Omega$, and so $w\equiv 0$ in $\Omega$.
\end{proof}

The approximation procedure in the above proof cannot be used to extend theorem \ref{W2,p quad} for unbounded $b$ and $d$, since in this case we do not have uniqueness results to infer that the limiting function is the same as the one we had started with.
However, knowing a priori that the solution is strong, we can obtain $W^{2,p}$ a priori estimates in the general case, as a kind of generalization of Nagumo's lemma (for instance, lemma 5.10 in \cite{Troiani}).

\begin{lem} \label{Nagumo}{$($Generalized Nagumo's lemma\,$)$}
Let $\Omega\subset\rn$ be a bounded $C^{1,1}$ domain. Let $F$ be a convex or concave operator in the $X$ entry, satisfying \ref{SCmu}, with $b,\, d\in L^p_+ (\Omega)$ for  $p>n$ and $\omega$ an arbitrary modulus.
Suppose that there exists $\theta>0$ such that \eqref{Htheta} holds for some $r_0>0$ and for all $x_0 \in \overline{\Omega}$.
Let $f\in L^p(\Omega)$, $\psi\in W^{2,p}(\Omega)$ and let $u\in W^{2,p}(\Omega)$ be a strong solution of
\begin{align*}
\left\{
\begin{array}{rllll}
F(x,u,Du,D^2 u) &=& f(x)&\mbox{in}&\Omega\\
u&=& \psi &\mbox{on}&\partial\Omega
\end{array}
\right.
\end{align*}
such that $\| u \|_{L^{\infty} (\Omega)} + \|f \|_{L^p (\Omega)} + \|\psi\|_{W^{2,p}(\Omega)} \leq C_1$.
Then we have
\begin{align}\label{est nagumo}
\|u\|_{W^{2,p}(\Omega)} \leq C\, \{ \| u \|_{L^{\infty} (\Omega)} + \|f \|_{L^p (\Omega)}  + \|\psi\|_{W^{2,p} (\Omega)} +\| d \|_{L^p (\Omega)}\,\omega(\| u \|_{L^{\infty} (\Omega)} ) \}
\end{align}
where $C$ depends on $r_0,n,p,\lambda,\Lambda,\mu,\| b \|_{L^p (\Omega)},C_1,\partial\Omega$ and $\mathrm{diam} (\Omega)$. The local case is analogous.

If $\mu = 0$, then the above constant $C$ does not depend on $C_1$.
\end{lem}

\begin{proof}
Note that, in particular, $u\in C^{1,\alpha}(\overline{\Omega})$ and satisfies $F(x,0,0,D^2 u)=g(x)$ a.e. in $\Omega$,
where $$g(x):=f(x)-F(x,u,Du,D^2u)+F(x,0,0,D^2u)\in L^p (\Omega),$$
since
$|F(x,u,Du,D^2u)-F(x,0,0,D^2u)|\leq b(x)|Du|+\mu |Du^2|+d(x)\,\omega(|u|) \in L^p(\Omega).$
Now, by theorem \ref{W2,p quad} (for $b,d,\mu,\gamma= 0$) and the proof there dealing with $C^{1,\alpha}$ estimates,
\begin{align*}
&\hspace{3cm}\|u\|_{W^{2,p}(\Omega)} \leq C\, \{ \| u \|_{L^{\infty} (\Omega)} + \|g \|_{L^p (\Omega)}  + \|\psi\|_{W^{2,p} (\Omega)} \}\\
&\leq C\,\{ \| u \|_{L^{\infty} (\Omega)} + \|f \|_{L^p (\Omega)} +\mu \|u\|_{C^1(\overline{\Omega})}^2 + \|\psi\|_{W^{2,p} (\Omega)}+\|b\|_{L^p(\Omega)}\|u\|_{C^1(\overline{\Omega})}+\|d\|_{L^p(\Omega)}\,\omega(\| u \|_\infty ) \}
\end{align*}
from where \eqref{est nagumo} follows.
\end{proof}

\section{The weighted eigenvalue problem}\label{First eigenvalue}

We start recalling some notations.

A subset $K\subset E$ of a Banach space is an ordered cone if it is closed, convex, $\lambda K\subset K$ for all $\lambda\geq0$ and $K\cap (-K)=\{0\}$. This cone induces a partial order on $E$, for
$ u,v\in E$, given by $u\leq v \Leftrightarrow v-u\in K.$
We say that $K$ is solid if $\mathrm{int}K\neq\emptyset$.
Further, a \textit{completely continuous} operator, defined in $E$, is continuous and takes bounded sets into precompact ones.

Following the construction of \cite{BR}, \cite{Quaas2004},  we have the following Krein-Rutman theorem for nonlinear operators --  see \cite{tese}.

\begin{teo}\label{KRquaas}{$($Generalized Krein-Rutman\,$)$}
Let $K\subset E$ be an ordered solid cone and let $T:K\rightarrow K$ be a completely continuous operator that is also
\begin{enumerate}[(i)]
\item positively 1-homogeneous, i.e. $T(\lambda u)=\lambda Tu$, for all $ \lambda\geq 0, u\in K$;
\item monotone increasing, i.e. for all $ u,v\in K$, $u\leq v$ we have $ Tu\leq Tv$;
\item strongly positive with respect to the cone, in the sense that  $T(K\setminus \{0\})\subset \mathrm{int}K$.
\end{enumerate}
Then $T$ has a positive eigenvalue $\alpha_1>0$ associated to a positive eigenfunction $w_1\in \mathrm{int}K$.\label{itemKpositive}
\end{teo}

Consider $\Omega\subset\rn$ a bounded $C^{1,1}$ domain. The application of Krein-Rutman is very standard for positive weights \cite{BR}, \cite{HessKato}. Let us recall its use when we have a fully nonlinear operator with unbounded coefficients.
About structure, we suppose
\begin{align}\mathcal{M}_{\lambda, \Lambda}^- (X-Y)-b(x)|p-q|-d(x)\,\omega ((r-s)^+) \leq F(x,r,p,X) - F(x,s,q,Y) \label{SC}\tag{SC}  \\ \leq \mathcal{M}_{\lambda, \Lambda}^+ (X-Y)+b(x)|p-q|+d(x)\,\omega ((s-r)^+)\;\; \textrm{ for } x\in \Omega \nonumber\end{align}
with $F(\cdot,0,0,0)\equiv 0$, where $0<\lambda \leq \Lambda$, $ b\in L^p_+ (\Omega)$, $ p>n$, $d\in L^\infty_+ (\Omega)$, $\omega$ a Lipschitz modulus.
Here, the condition over the zero order term in \eqref{SC} means that $F$ is proper, i.e.\ decreasing in~$r$.

Consider $E=C_0^1(\overline{\Omega})$ and the usual ordered solid cone $K=\{u\in E;\,u\geq 0 \textrm{ in }\Omega\}$ in $E$.

Let $c(x)\in L^p_+(\Omega)$ with $c>0$ in $\overline{\Omega}$, $p> n$.
As the operator on $K$, we take $T=-F^{-1}\circ c$ in the sense that $U=Tu$ iff $U$ is the unique $L^n$-viscosity solution of the Dirichlet problem
\begin{align}\label{Tu}\tag*{$(T_u)$}
\left\{
\begin{array}{rclcc}
F(x,U,DU,D^2U)&=& -c(x)u & \mbox{in} & \Omega\\
U &=& 0 &\mbox{on}& \partial\Omega
\end{array}
\right.
\end{align}
where $F$ satisfies the following hypotheses
\begin{align}\label{HF} \tag{$H$}
\left\{
\begin{array}{l}
\textrm{there exists } \theta>0 \textrm{ such that } (\overline{H})_\theta \textrm{ holds for for all } x_0\in \overline{\Omega}, \\
\eqref{SC} \textrm{ and } \eqref{ExistUnic T bem definido} \textrm{ hold},\;\,F(x,tr,tp,tX)=tF(x,r,p,X) \textrm{ for all } t\geq 0.
\end{array}
\right.
\end{align}
Here, hypothesis \eqref{ExistUnic T bem definido} means the solvability in $L^n$- sense with data in $L^p$, i.e. for any $f\in L^p(\Omega)$,
\begin{align}\label{ExistUnic T bem definido} \tag{$S$}
\textrm{there exists a unique }  u\in C(\overline{\Omega})\; L^n\textrm{-viscosity solution of }
F[u]=f(x) \textrm{ in }\Omega; \;u=0 \textrm{ on }\partial\Omega.
\end{align}
Of course, Pucci's extremal operators
$$\mathcal{L}^\pm [u]:=\mathcal{M}^\pm (D^2 u)\pm b(x)|Du|\pm d(x)\omega(u^\mp), \;\;b,\,d\in L^p_+(\Omega),$$
where $\omega$ is a Lipschitz modulus, are particular examples of $F$ satisfying \eqref{HF}. Indeed, recall that proposition \ref{W2,p solv} provides a strong solution $u\in W^{2,p}(\Omega)\subset W^{2,n}(\Omega)$, which is an $L^n$-viscosity solution by proposition \ref{Lpiffstrong.quad}. Furthermore, since it is unique among $L^p$-viscosity solutions, it is also unique among $L^n$-viscosity ones. Here all the coefficients can be unbounded.
Observe that \eqref{ExistUnic T bem definido} and $(\overline{H})_\theta$ also holds when $F$ is a uniformly continuous operator in $x$ satisfying the growth conditions in \cite{Arms2009} (see also \cite{IY}), in this case concerning $C$-viscosity notions of solutions.

On the other hand, $(\overline{H})_\theta$, \eqref{SC} and \eqref{ExistUnic T bem definido} are completely enough to ensure existence, uniqueness and $C^{1,\alpha}$ global regularity and estimates for the problem \ref{Tu} from Theorem \ref{C1,alpha regularity estimates geral}, which in turn implies that the operator $T$ is well defined and completely continuous.

Furthermore, $T=-F^{-1}\circ c$ is strictly positive with respect to the cone, thanks to SMP and Hopf.
In general, without the strict positiveness of $c$ in $\Omega$ there is no guarantee on this property, i.e., under $c\geq 0$ and $c\not\equiv 0$ in $\Omega$, we only obtain that $T(K\setminus\{0\})\subset K$.

Notice that $T=-F^{-1}\circ c\,$ has an eigenvalue $\alpha_1>0$ associated to the positive eigenfunction $\varphi_1$ if and only if $\varphi_1$ is an $L^n$-viscosity solution of
\begin{align*}
\left\{
\begin{array}{rclcc}
F[\varphi_1]+1/\alpha_1 \, c(x) \varphi_1 &=& 0 &\mbox{in} & \Omega \\
\varphi_1 &>& 0 &\mbox{in} &\Omega \\
\varphi_1 &=& 0 &\mbox{on} &\partial\Omega\, .
\end{array}
\right.
\end{align*}

For any $c\in L^p(\Omega)$ with $p> n$ and $F$ satisfying \eqref{HF}, we can define, as in \cite{BNV}, \cite{QB},
\begin{align*}
\lambda_1^\pm=\lambda_1^\pm\,(F(c),\Omega)&=\sup\left\{ \lambda>0; \; \Psi^\pm(F(c),\Omega,\lambda)\neq \emptyset\right\}
\end{align*}
where
$\Psi^\pm (F(c),\Omega,\lambda):=\left\{ \psi\in C(\overline{\Omega}); \; \pm\psi>0 \textrm{ in }\Omega,\; \pm (F[\psi]+\lambda c(x)\psi )\leq 0 \textrm{ in }\Omega \right\}$;
with inequalities holding in the $L^n$-viscosity sense. Notice that, by definition,
$\lambda_1^\pm (G(c),\Omega)=\lambda_1^\mp (F(c),\Omega),$
where $G(x,r,p,X):=-F(x,-r,-p,-X)$.

With an approximation procedure by positive weights given by Krein-Rutman theorem as above, for $F$ satisfying \eqref{HF}, we obtain existence of eigenvalues with nonnegative weight.

\begin{teo}\label{exist eig for F c geq 0}
Let $\Omega\subset\rn$ be a bounded $C^{1,1}$ domain, $c\in L^p(\Omega)$, $c\gneqq 0$ for $p>n$ and $F$ satisfying \eqref{HF} for $b,\, d\in L_+^\infty (\Omega)$. Then $F$ has two positive weighted eigenvalues $\alpha_1^\pm>0$ corresponding to normalized and signed eigenfunctions $\varphi_1^\pm\in C^{1,\alpha}(\overline{\Omega})$ that satisfies
\begin{align} \label{eq exist eigen F c lambda1+}
\left\{
\begin{array}{rclcc}
F[\varphi_1^\pm]+\alpha_1^\pm c(x) \varphi_1^\pm &=& 0 &\mbox{in} & \Omega \\
\pm \varphi_1^\pm &>& 0 &\mbox{in} &\Omega \\
\varphi_1^\pm &=& 0 &\mbox{on} &\partial\Omega
\end{array}
\right.
\end{align}
in the $L^p$-viscosity sense, with $\max_{\overline{\Omega}}\,(\pm \varphi_1^\pm) =1$.

If, moreover, the operator $F$ has $W^{2,p}$ regularity of solutions $($in the sense that every $u\in C(\overline{\Omega})$ which is an $L^p$-viscosity solution of $F[u]=f(x)\in L^p(\Omega)$, $u=0$ on $\partial\Omega$, satisfies $u\in W^{2,p}(\Omega))$, then $\alpha_1^\pm=\lambda^\pm_1$ and the conclusion is valid also for $b\in L^p_+(\Omega)$.
\end{teo}

Notice that we obtain positive eigenvalues because $F$ is proper. For general existence related to nonproper operators see the script in \cite{QB} for bounded coefficients.
We also stress that, without regularity assumptions on the domain, it is still possible to obtain the existence of an eigenpair, as in \cite{BNV} and \cite{QB}; in such cases the eigenfunction belongs to $C^{1,\alpha}_{\mathrm{loc}}(\Omega)\cap C(\overline{\Omega})$ by using $C^{1,\alpha}$ local regularity instead of the global one.

We start proving some auxiliary results which take into account the unboundedness of $c$.

\begin{prop}\label{th4.1 QB} Let $u,v\in C(\overline{\Omega})$ be $L^n$-viscosity solutions of
\begin{align}\label{eq th1.4 2case ineq}
\left\{
\begin{array}{rclcc}
F[u]+c(x)u &\geq & 0 & \mbox{in} &\Omega \\
u &<& 0& \mbox{in} & \Omega
\end{array}
\right. ,\quad
\left\{
\begin{array}{rcll}
F[v]+c(x)v &\leq & 0 &\mbox{in} \;\;\; \Omega \\
v &\geq & 0 & \mbox{on} \;\; \partial\Omega\\
v(x_0) &< & 0  &x_0 \in\Omega
\end{array}
\right.
\end{align}
with $F$ satisfying \eqref{HF}, $c\in L^p(\Omega)$, $p>n$. Suppose one, $u$ or $v$, is a strong solution. Then, $u=tv$ for some $t>0$.
The conclusion is the same if $F[u]+c(x)u\leq 0$, $F[v]+c(x)v\geq 0 $ in $\Omega$, with $u>0$ in $\Omega$, $v\leq 0$ on $\partial\Omega$ and $v(x_0)>0$ for some $x_0 \in\Omega$.
\end{prop}

For the proof of proposition \ref{th4.1 QB}, as in \cite{Arms2009}, \cite{BNV}, \cite{QB}, we need the following consequence of ABP, which is MP for small domains.

\begin{lem}\label{MP small domain}
Assume $F$ satisfies \eqref{SC} and $c\in L^p(\Omega)$, $p>n$. Then there exists $\varepsilon_0>0$, depending on $n,p,\lambda,\Lambda$, $\|b\|_{L^p(\Omega)}$, $\|c^+\|_{L^p(\Omega)}$ and $\mathrm{diam}(\Omega)$, such that if $|\Omega|\leq \varepsilon_0$ then any $u\in C(\overline{\Omega})$ which is an $L^n$-viscosity solution of
\begin{align}\label{MP small domain eq u}
\left\{
\begin{array}{rclcc}
F[u]+c(x)u &\geq & 0 &\mbox{in} &\Omega\\
 u &\leq & 0 &\mbox{on} & \partial\Omega
\end{array}
\right.
\end{align}
satisfies $u\leq 0$ in $\Omega$. Analogously, any $v\in C(\overline{\Omega})$ that is an $L^n$-viscosity solution of $F[v]+c(x)v\leq 0$ in $\Omega$, with $v\geq 0$ on $\partial\Omega$, is such that $v\geq 0$ in $\Omega$ provided $|\Omega|\leq \varepsilon_0$.
\end{lem}

\begin{proof}
Assume $u$ satisfies \eqref{MP small domain eq u}. In order to obtain a contradiction, suppose that $\Omega^+:=\{u>0\}$ is not empty. By \eqref{SC}, we have that $u$ is an $L^n$-viscosity solution of
\begin{center}
$\mathcal{M}^+(D^2u)+b(x)|Du|\geq \mathcal{M}^+(D^2u)+b(x)|Du|-c^-(x)u\geq -c^+(x)u$ \; in $\Omega^+$.
\end{center}
Hence, ABP gives us that
$$\sup_{\Omega^+}u\leq C_1 \,\textrm{diam}(\Omega)\, \|c^+\|_{L^n(\Omega)}\sup_{\Omega^+}u \leq C_1 \,\textrm{diam}(\Omega)\, |\Omega|^{1-\frac{n}{p}}\|c^+\|_{L^p(\Omega)}\sup_{\Omega^+}u. $$
Then we choose $\epsilon_0>0$ such that $C_1 \,\textrm{diam}(\Omega)\, \varepsilon_0^{1-\frac{n}{p}}\|c^+\|_{L^p(\Omega)}\leq 1/2$ to produce a contradiction. If $v$ is a supersolution it is similar, by using ABP in the opposite direction.
\qedhere{\,\textit{Lemma \ref{MP small domain}.}}
\end{proof}

\begin{proof}[\textit{Proof of Proposition \ref{th4.1 QB}.}]
We are going to prove the first case, since the second is analogous.
Let $u,v$ be $L^n$-viscosity solutions of \eqref{eq th1.4 2case ineq}. Say both are strong, otherwise just use test functions for one of them and read all inequalities below in the $L^n$-viscosity sense.
Set $z_t:=tu-v$ for $t>0$. Then, using 1-homogeneity and \eqref{SC}, we have that $z_t$ is a solution of
\begin{align}\label{eq th 1.4 sem DF}
\mathcal{M}^+(D^2 z_t)&+b(x)|Dz_t|+d(x)\omega((-z_t)^+)+c(x)z_t \geq F[tu]-F[v]+c(x)z_t \nonumber\\
&=t\,\{F[u]+c(x)u\}-\{F[v]+c(x)v\}\geq 0 \;\;\textrm{in }\Omega\, .
\end{align}

Let $K$ be a compact subset of $\Omega$ such that $x_0\in K$ and MP lemma \ref{MP small domain} holds for $\Omega\setminus K$. Further, let $t_0>0$ be large enough such that $z_{t_0}\leq 0$ in $K$. In fact, this $t_0$ can be taken as $\min_K v / \max_K u >0$, since $u<0$ in $K$ and $\min_K v\leq v(x_0)<0$.
Then, since $z_{t_0}\leq 0$ in $\partial\, (\Omega\setminus K)\subset \partial\Omega\cup\partial K$, we obtain from lemma \ref{MP small domain} that $z_{t_0}\leq 0$ in $\Omega\setminus K$ and so in $\Omega$.

Define $\tau:=\inf \{t>0;\;z_t\leq 0\;\textrm{ in }\Omega\}\geq t_0>0$.
Hence, using \eqref{SC}, we have that $z_\tau$ is a nonpositive solution of
$
\mathcal{L}^-[-z_\tau]+\{c(x)-d(x)\,\omega (1)\}(-z_\tau) \leq  0 \;\;\textrm{ in } \Omega
$
and so by SMP we have either $z_\tau\equiv 0$ or $z_\tau <0$ in $\Omega$. In the first case we are done. Suppose, then, $z_\tau <0$ in $\Omega$ in order to obtain a contradiction.

Next we choose some $\varepsilon >0$ such that $z_{\tau-\varepsilon}<0$ in $K$. Indeed, we can take, for example, $\varepsilon =\min \{-\min_K z_\tau / (2\|u\|_{L^\infty (K)}), \tau/2 \}$, which implies, as in \cite{Patrizi},
$$z_{\tau -\varepsilon}=z_\tau -\varepsilon u\leq \min_K z_\tau +\varepsilon\|u\|_{L^\infty (K)}<0 \textrm{ in }K.$$
In particular, $z_t$ satisfies \eqref{eq th 1.4 sem DF} for $t=\tau -\varepsilon >0$ . Thus, $z_{\tau-\varepsilon}\leq 0$ by MP in $\Omega \setminus K$.  By SMP, $z_{\tau-\varepsilon}<0$ in $\Omega$, which contradicts the definition of $\tau$ as an infimum.
\qedhere{\textit{Proposition \ref{th4.1 QB}. }}
\end{proof}

The next result was first introduced in \cite{BNV} and extended in \cite{QB} to nonlinear operators. When we add an unbounded weight $c$, all we need is its positiveness on a subset of positive measure in order to obtain a bound from above on $\lambda_1$.

\begin{lem} \label{boundedness eig QB}
Suppose \eqref{HF} with $b,\, d\in L_+^\infty (\Omega)$. If $c\geq \delta>0$ a.e. in $B_R\subset\subset\Omega$, for $R\leq 1$, then
$$\lambda_1^\pm (F(c),\Omega)\leq\frac{C_0}{\delta R^2}$$
for a positive constant $C_0$ that depends on $n,\lambda,\Lambda,R,\|b\|_{L^\infty(\Omega)} $ and $\omega (1)\|d\|_{L^\infty(\Omega)}$.

If, moreover, $F$ has no term of order zero (i.e. $d$ or $\omega$ is equal to zero), then $R$ can be any positive number. On the other hand, if $b\equiv 0$, then $C_0$ does not depend on $R$.
\end{lem}

\begin{proof}
Observe that $\lambda_1^\pm(F(c),\Omega) \leq \lambda_1^\pm (F(c),B_R)$ by definition.

Consider, as in \cite{BNV} and \cite{QB}, the radial function $\sigma (x):=-(R^2-|x|^2)^2<0$ in $B_R$.
Let us treat the case of $\lambda_1^-$, since for $\lambda_1^+$ it is just a question of looking at $-\sigma$.

Suppose, in order to obtain a contradiction, that there exists some $\lambda>\frac{C_0}{\delta R^2}$ such that $\Psi^-(F(c),\Omega,\lambda)\neq\emptyset$, i.e. let $\psi\in C(\overline{\Omega})$ be a negative $L^n$-viscosity solution of $F[\psi]+\lambda c(x)\psi\geq 0$ in $\Omega$; also of $F[\psi]+\frac{C_0}{\delta R^2} c(x)\psi\geq 0$ in $B_R$.

\begin{claim} \label{claim sigma}
We have $F[\sigma]+ \frac{C_0}{\delta R^2} c(x)\, \sigma \leq 0$ a.e. in $B_R$.
\end{claim}

\begin{proof}
Say, for example, $b(x)\leq\gamma$ and $d(x)\leq \eta$ a.e., then it holds (see \cite{BNV} or \cite{tese})
\begin{align*}
\frac{F[\sigma]}{\sigma}\geq \frac{8\lambda\, |x|^2}{(R^2-|x|^2)^2} -\frac{4n\Lambda}{R^2-|x|^2}-\frac{4\gamma R}{R^2-|x|^2}-\eta \,\omega (1) \;\textrm{ a.e. in } B_R\, .
\end{align*}
Hence, if we take
$\alpha=({n\Lambda +\gamma R})(2\lambda+n\Lambda+\gamma R)^{-1}\in (0,1)$,
we have two cases.
\vspace{0.1cm}

$(a)$ $|x|^2\geq \alpha R^2$: From construction,
${F[\sigma]}/{\sigma}\geq - \eta\, \omega (1)\geq-{\eta\, \omega (1)}\,c(x)/({\delta R^2} ).$

$(b)$ $|x|^2\leq \alpha R^2$: In this case we just bound the first term by zero; the others are such that
${F[\sigma]}/{\sigma}\geq - {4(n\Lambda +\gamma R)}/({(1-\alpha)R^2}) - \eta\, \omega (1)\geq - {C_0}\,c(x)/({\delta R^2}).$
\qedhere{\textit{Claim \ref{claim sigma}.}}
\end{proof}

Now we apply proposition \ref{th4.1 QB}, since $\sigma\in C^2(\overline{B}_R)$, obtaining that $\psi = t\sigma$, for some $t>0$. However, this is not possible, since $\psi <0$ on $\partial B_R\subset\Omega$ while $\sigma =0$ on $\partial B_R\,$.
\qedhere{\textit{Lemma \ref{boundedness eig QB}.} }
\end{proof}

Moving to the last statement in theorem \ref{exist eig for F c geq 0}, we first prove an eigenvalue bound that takes into account an unbounded $b$, when the weight is a continuous and positive function in $\overline{\Omega}$.
Note that, in this case, theorem \ref{KRquaas} gives us a pair $\alpha_1>0$ and $\varphi_1\in C^{1}(\overline{\Omega})$ such that
\begin{align} \label{eq exist eigen G 1}
\left\{
\begin{array}{rclcc}
G[\varphi_1]+\alpha_1 \,c(x) \varphi_1 &= &0  &\mbox{in} & \Omega \\
\varphi_1 &> &0  &\mbox{in} &\Omega \\
\varphi_1 &=& 0 &\mbox{on} &\partial\Omega
\end{array}
\right.
\end{align}
in the $L ^n$-viscosity sense, with $\max_{\overline{\Omega}}\,\varphi_1=1$ and $0<\alpha_1\leq \lambda_1^+(G(c),\Omega)=\lambda_1^-(F(c),\Omega).$

The following lemma is a delicate point in our construction of an eigenpair. It states that $\alpha_1$ in \eqref{eq exist eigen G 1} is bounded, and this does not seem to be a consequence of the usual methods for bounding a first eigenvalue, such as the one in lemma \ref{boundedness eig QB}. Instead, we use the classical blow-up method \cite{GS} of Gidas and Spruck.

\begin{lem}\label{limit unbounded b cont c}
Let $c\in C(\overline{\Omega})$, $c>0$ in $\overline{\Omega}$ and $G$ satisfying \eqref{HF} with $b\in L^p_+(\Omega)$, $d\in L_+^\infty (\Omega)$. Let $\alpha_1$ and $\varphi_1$ as in \eqref{eq exist eigen G 1}. Then $\alpha_1\leq C$, for $C=C(n,\lambda,\Lambda,\Omega, \|b\|_{L^p(\Omega)},\omega (1)\|d\|_{L^\infty(\Omega)})$. 
\end{lem}

\begin{proof}
If the conclusion is not true, then exists a sequence $b_k\in L^\infty_+ (\Omega)$, with $\|b_k\|_{L^p(\Omega)}\leq C$, $\|b_k\|_{L^\infty (\Omega)}\rightarrow +\infty$ and the respective eigenvalue problem
\begin{align} \label{eq exist eigen Gk c epsilon}
\left\{
\begin{array}{rclcc}
G_k\,[\varphi_k]+\alpha_1^k \,c(x) \varphi_k &= &0  &\mbox{in} & \Omega \\
\varphi_k &> &0  &\mbox{in} &\Omega \\
\varphi_k &=& 0 &\mbox{on} &\partial\Omega
\end{array}
\right.
\end{align}
in the $L^n$-viscosity sense, with $\max_{\overline{\Omega}}\,\varphi_k=1$ for all $k\in\n$ and $\alpha_1^k\rightarrow +\infty$ as $k\rightarrow +\infty$, where $G_k$ is a fully nonlinear operator satisfying $(H)_k\,$, i.e. \eqref{HF} for $b_k$ and $d_k$. Say $d_k\leq \eta$ and $\max_{\overline{\Omega}}\varphi_k=\varphi_k(x_0^k)$ for $x_0^k\in\Omega$. Then, $x_0^k\rightarrow x_0\in \overline{\Omega}$ as $k\rightarrow +\infty$, up to a subsequence.

\vspace{0.1cm}
\textit{Case 1:} $x_0\in\Omega$. Let $2\rho=\mathrm{dist}(x_0,\partial\Omega)>0$ and notice that $x_0^k\in B_\rho (x_0)$ for all $k\geq k_0$.
Set $r_k={(\alpha_1^k)}^{-{1}/{2}}$ and define $\psi_k(x)=\varphi_k(x_0^k+r_k x)$. Thus, $\psi_k$ is an $L^n$ (so $L^p$) viscosity solution of
\begin{align*}
\widetilde{G}_k(x,\psi_k,D\psi_k,D^2\psi_k)+c_k (x)\psi_k(x)=0 \quad \textrm{ in }\; \widetilde{B}_k:=B_{\rho/{r_k}}(0)
\end{align*}
where ${c}_k (x):=c(x_0^k+r_k x)$, $\widetilde{G}_k(x,r,p,X):=r^2_k\, G_k(x_0^k+r_k x,r,p/{r_k},X/{r_k^2})$ satisfies $(\widetilde{H})_k$, i.e. \eqref{HF} for $\widetilde{b_k}$ and $\eta_k$, where $\widetilde{b}_k(x):=r_k \,b_k(x_0^k+r_k x) $ and ${\eta}_k=r_k^2 \,\eta$. Notice that $b_k$ and $\eta_k$ converge locally to zero in $L^p(\widetilde{B}_k)$ as $k\rightarrow +\infty$, since $p>n$.

Furthermore, $\sup_{\widetilde{B}_k}\psi_k=\psi_k(0)=1$ for all $k\in\n$ and $B_R(0)\subset\subset \widetilde{B}_k$ for large $k$, for  any fixed $R>0$.
By theorem \ref{C1,alpha regularity estimates geral} we have that $\psi_k$ is locally in $C^{1,\alpha}$ and satisfies the estimate
\begin{align*}
\|\psi_k\|_{C^{1,\alpha}(\overline{B}_R(0))}\leq C_k \|\psi_k\|_{L^\infty (\widetilde{B_k})}\leq C,
\end{align*}
since $\psi_k$ attains its maximum at $0$ and $C_k$ depends only on the $L^p$-norm of the coefficients $b_k$ and $c_k$, which are uniformly bounded in there.
Hence, by compact inclusion we have that there exists $\psi\in C^1(\overline{B}_R(0))$ such that $\psi_k\rightarrow \psi$ as $k\rightarrow +\infty$, up to a subsequence. Doing the same for each ball $B_R(0)$, for every $R>0$, we obtain in particular that $\psi_k\rightarrow \psi$ in $L^{\infty}_{\mathrm{loc}}(\rn)$, by using the uniqueness of the limit for $\psi_k$ in the smaller balls.

Using stability (proposition \ref{Lpquad} together with observation \ref{Lpquadencaixados}) and the continuity of $c$, we have that $\psi$ is an $L^p$-viscosity solution of $J(x,D^2\psi)+{c} (x_0)\psi=0$ in $\rn$ for some measurable operator $J$ still satisfying \eqref{HF} with coefficients of zero and first order term, $d$ and $b$, equal to zero.
Also,  $\psi(0)=1$ and $\psi>0$ in $\rn$ by SMP.
This implies that $1\leq \lambda_1^+ (J({c(x_0)}),B_R)\leq \frac{C_0}{c(x_0)R^2}$ for all $R>0$, which gives a contradiction when we take $R\rightarrow +\infty$.

\vspace{0.1cm}
\textit{Case 2:} $x_0\in\partial\Omega$.
By passing to new coordinates, that come from the smoothness property of the domain $\partial\Omega\in C^{1,1}$, we can suppose that $\partial\Omega\subset \{x_n=0\}$ and $\Omega\subset\{ x_n>0\}$. 

Set $\rho_k=\mathrm{dist}(x_0^k,\partial\Omega)=x_0^k\cdot e_n=x_{0,n}^k\,$, where $e_n=(0,\ldots,0,1)$, $x_{0}^k=(x_{0,1}^k,\ldots,x_{0,n}^k)$.
Analogously, consider $\psi_k (y)$ in $y\in B_{\rho_k/{r_k}}(0)$ and the respective equation $\widetilde{G}_k$ as in case 1.
Thus, we have for $x,y$ satisfying $r_k y=x-x_0^k\,$, that the set $\{x_n>0\}$ is equivalent to $A_k:=\{y_n=(x-x_0^k)\cdot e_n / {r_k}>-\rho_k/{r_k}\}$. Now we need to analyze the behavior of the set $A_k$ when we take the limit as $k\rightarrow +\infty$.

We first claim that $\rho_k/{r_k}$ is bounded below by a constant $C_1>0$, which means that $A_k$ does not converge to $\{y_n>0\}$.
This is an easy consequence of our $C^{1,\alpha}$ boundary regularity and estimates in a half ball, applied to $\psi_k$ and $\widetilde{G}_k$.
Indeed, since  $\|D \psi_k\|_{L^\infty (B_r^+(0))}\leq C$, then $1=|\varphi (x_0^k)-\varphi(\bar{x}_0^k)|=|\psi_k(0,0)-\psi_k(0,-\rho_k/{r_k})|\leq C \rho_k/{r_k}$, with $\bar{x}_0^k=(x_{0,1}^k,\ldots,x_{0,n-1}^k,0)\in \partial\Omega$ and fixed $r>0$, from where we obtain the desired bound.

Next observe that we have two possibilities about the fraction $\rho_k/ {r_k}$, either it converges to $+\infty$ or it is uniformly bounded. In the first one, $A_k\rightarrow\rn$ and we finish as in case 1. In the second, $A_k\rightarrow \{y_n> \varrho\}$, $\varrho\in (0,+\infty)$, by passing to a subsequence, and the proof carries on as in the case 1, since we have a smooth domain that contains a ball with radius $R=(2\,C_0/c(x_0)\,)^{1/2}$; this derives the final contradiction.
\end{proof}

\begin{lem}\label{limit final}
Let $c\in L^p(\Omega)$, $c\geq \delta$ in $B_R$ for some $B_R\subset\subset\Omega$ and $F$ satisfying \eqref{HF}, then
$$\lambda_1^\pm (F(c),\Omega) \leq \frac{\lambda_1^\pm (F(1),B_R)}{\delta}. $$
\end{lem}

\begin{proof}
Let us prove the $\lambda^+_1$ case; for $\lambda_1^-$ we use $G$ instead of $F$. We already know that both quantities are nonnegative, by the properness of the operator $F$. Hence, it is enough to verify that $\mathcal{A}\cap \{\lambda\geq 0\} \subset\mathcal{B}/\delta \cap \{\lambda\geq 0\}$, where
$$
\lambda_1^+(F(c),\Omega)=\sup_\mathcal{A} \lambda = \sup_{\mathcal{A}\cap \{\lambda\geq 0\}} \lambda\; ,\;\;  \lambda_1^+(F(1),B_R)=\sup_\mathcal{B} \lambda = \sup_{\mathcal{B}\cap \{\lambda\geq 0\}} \lambda
$$
as defined before. Let $\lambda\in \mathcal{A}\cap \{\lambda\geq 0\}$, then there exists $\psi\in C(\overline{\Omega})$ a nonnegative $L^n$-viscosity solution of $F[\psi]+c(x)\lambda\psi \leq 0$ in $\Omega$. Then, $\psi$ is also a nonnegative $L^n$-viscosity solution of $F[\psi]+\delta\lambda\psi \leq 0$ in $B_R$\,, from where $\delta\lambda\in \mathcal{B}$.
\end{proof}

\begin{proof}[\textit{Proof of Theorem \ref{exist eig for F c geq 0}.}]
First, from the fact that $c>0$ in a set of positive measure,  there exists $\delta>0$ such that $\{c\geq \delta\}$ is a nontrivial set. In fact, if this was not true, i.e. if $|\{c\geq\delta\}|=0$ for all $\delta$, then $\{c>0\}=\bigcup_{\delta>0} \{c\geq\delta\}$ would have measure zero, as the union of such sets, contradicting the hypothesis.
Namely, then, $c\geq \delta>0$ a.e. in some ball $B_R\subset\subset\Omega$.

Let us prove the $\lambda_1^-$ case, applying Krein-Rutman results to $G$; for $\lambda_1^+$ replace $G$ by $F$.

Let $\varepsilon\in (0,1)$ and define $c_\varepsilon:=c+\varepsilon >0$ in $\Omega$, for all $\varepsilon$.
From theorem \ref{KRquaas}, we obtain the existence of pairs $\alpha_1^\varepsilon>0$ and $\varphi_1^\varepsilon\in C^{1}(\overline{\Omega})$ such that
\begin{align} \label{eq exist eigen G c epsilon}
\left\{
\begin{array}{rclcc}
G[\varphi_1^\varepsilon]+\alpha_1^\varepsilon \,c_\varepsilon(x) \varphi_1^\varepsilon &= &0  &\mbox{in} & \Omega \\
\varphi_1^\varepsilon &> &0  &\mbox{in} &\Omega \\
\varphi_1^\varepsilon &=& 0 &\mbox{on} &\partial\Omega
\end{array}
\right.
\end{align}
with $\max_{\overline{\Omega}}\,\varphi_1^\varepsilon=1$ for all $\varepsilon\in (0,1)$. Then,
\begin{align}\label{cota lambda1 b,d limit}
0<\alpha_1^\varepsilon\leq \lambda_1^+(G(c_\varepsilon),\Omega)=\lambda_1^-(F(c_\varepsilon),\Omega)\leq \frac{C_0}{\delta R^2}\,\;\textrm{ for all }\varepsilon\in (0,1).
\end{align}

Next, $\alpha_1^\varepsilon\rightarrow\alpha_1 \in [0,C_0/{\delta R^2}]$ up to a subsequence.
Then, applying $C^{1,\alpha}$ global regularity and estimates (theorem \ref{C1,alpha regularity estimates geral}) in the case $\mu=0$ (recall again that $L^n$-viscosity solutions are $L^p$-viscosity for $p>n$), by considering $\alpha_1^\varepsilon \,c_\varepsilon (x)\,\varphi_1^\varepsilon \in L^p (\Omega)$ as the right hand side, we obtain
\begin{align*}
\|\varphi_1^\varepsilon\|_{C^{1,\alpha}(\overline{\Omega})} &\leq C\,\{\,\|\varphi_1^\varepsilon\|_{L^\infty(\Omega)} +\alpha_1^\varepsilon\, \|c_\varepsilon\|_{L^p(\Omega)} \,\|\varphi_1^\varepsilon \|_\infty +1\,\}
\leq C\, C_1 \,(\|c\|_{L^p(\Omega)} +1) \,\} \leq C.
\end{align*}
Hence the compact inclusion $C^{1,\alpha}(\overline{\Omega})\subset C^1(\overline{\Omega})$ yields $\varphi_1^\varepsilon \rightarrow \varphi_1\in C^1(\overline{\Omega})$, up to a subsequence. Of course this implies that $\max_{\overline{\Omega}}\,\varphi_1=1$, $\varphi_1\geq 0$ in $\Omega$ and $\varphi_1=0$ on $\partial\Omega$.

Since $c_\varepsilon\rightarrow c$ in $L^p(\Omega)$ as $\varepsilon\rightarrow 0$, by proposition \ref{Lpquad} we have that $\varphi_1$ is an $L^p$-viscosity solution of $G[\varphi_1]+\alpha_1 c(x)\varphi_1=0$ in $\Omega$, which allows us to apply $C^{1,\alpha}$ regularity again to obtain that $\varphi_1\in C^{1,\alpha}(\overline{\Omega})$.

Using now that $\varphi_1$ is an $L^p$-viscosity solution of
$\mathcal{L}^-[\varphi_1]-(d(x)\omega (1)-\alpha_1 c(x))\,\varphi_1 \leq 0$ in $ \Omega$, together with SMP, we have that $\varphi_1>0$ in $\Omega$, since $\max_{\overline{\Omega}}\,\varphi_1=1$.
Moreover, we must have $\alpha_1>0$, because the case $\alpha_1=0$ would imply that $\varphi_1$ is an $L ^p$-viscosity solution of $\mathcal{L}^+[\varphi_1] \geq 0$ in $\Omega\cap \{\varphi_1>0\}$ (since $F$ is proper, and so $G$) which, in turn, would give us $\varphi_1\leq 0$ in $\Omega$, by ABP. Thus, the existence property is completed.

In order to conclude that, under $W^{2,p}$ regularity assumptions over $F$, the $\alpha_1$ obtained is equal to $\lambda_1^-=\lambda_1^-(F(c),\Omega)$, related to $\varphi_1^-=\varphi_1^-(F(c),\Omega)=-\varphi_1<0$ in $\Omega$, we have to work a little bit more, as in proposition 4.7 in \cite{QB}.

We already have that $\alpha_1\leq\lambda_1^-$. Suppose by contradiction that $\alpha_1<\lambda_1^-$.
By definition of $\lambda_1^-$ as a supremum, we know that $\alpha_1$ cannot be an upper bound, i.e. there exists $\lambda>0$ such that $\Psi^-(F(c),\Omega,\lambda)\neq\emptyset$ and $\alpha_1<\lambda\leq\lambda_1^-$. Then we obtain $\psi\in C(\overline{\Omega})$ such that
$F[\psi]+\lambda c(x)\psi\geq 0$ in $\Omega$
in the $L^n$-viscosity sense, with $\psi<0$ in $\Omega$.
Now, since $c\gneqq 0$, we have $c(x)(\lambda-\alpha_1)\gneqq 0$. Next, $\psi$ is an $L^n$-viscosity solution of
\begin{align}\label{contradction exist eig}
F[\psi]+\alpha_1 \,c(x)\psi\gneqq F[\psi]+\lambda\, c(x)\psi\geq 0\;\textrm{ in }\Omega\, .
\end{align}
Then, under $W^{2,p}$ regularity, we have that $\varphi_1^-\in W^{2,p}(\Omega)\subset W^{2,n}(\Omega)$ is a strong solution of
\begin{align*}
\left\{
\begin{array}{rclcc}
F[\varphi_1^-]+\alpha_1 \,c(x) \varphi_1^- &=& 0 &\mbox{in} & \Omega \\
\varphi_1^- &<& 0 &\mbox{in} &\Omega \\
\varphi_1^- &=& 0 &\mbox{on} &\partial\Omega\, .
\end{array}
\right.
\end{align*}
Applying proposition \ref{th4.1 QB} we obtain that $\psi=t\varphi^-_1$ for some $t>0$; but this contradicts the strict inequality in \eqref{contradction exist eig}. Thus, we must have $\alpha_1=\lambda_1^-$. The case of $\lambda_1^+$ is completely analogous, by reversing the inequalities.

From this last paragraph, under $W^{2,p}$ regularity of the solutions, the only possibility to $\alpha_1$ is to coincide with $\lambda_1$.
Therefore, by using lemmas \ref{limit unbounded b cont c} (with $c\equiv 1$) and \ref{limit final}, we obtain that
$\lambda_1^-(F(c_\varepsilon),\Omega)\leq C_1/\delta$, for all $\varepsilon\in (0,1)$,
where $C_1$ depends on $n,\lambda,\Lambda,R$, $\|b\|_{L^p(\Omega)}$ and $\omega (1)\|d\|_{L^\infty(\Omega)}$.
Thus, we carry on this bound on $\lambda_1$, instead of \eqref{cota lambda1 b,d limit}, in the limiting procedure, in order to get the desired existence result for $b\in L^p(\Omega)$.
\qedhere{\textit{Theorem \ref{exist eig for F c geq 0}.}}
\end{proof}

\bibliography{bibtex}
\bibliographystyle{abbrv}

\end{document}